\documentclass{article}[11pt]

\usepackage{amsthm,amsfonts,amssymb,amsmath,tikz}
\usepackage[english]{babel}
\usepackage{tikz}
\usepackage{amssymb}
\usepackage[utf8]{inputenc}
\usepackage{amsmath}
\usepackage{amsthm}
\usepackage{enumerate}
\usepackage{caption}
\usepackage{graphicx}
\usetikzlibrary[shapes.misc]
\usetikzlibrary[shapes.geometric]

\newtheorem{theo}{Theorem}
\newtheorem{prop}[theo]{Proposition}

\newtheorem{lem}[theo]{Lemma}

\newtheorem{ques}{Question}

\newcommand{\pch}{\chi_{\rho}}

\tikzstyle{vertex}=[circle, draw, inner sep=0pt, minimum size=6pt]

\begin{document}

\title{Packing colorings of subcubic outerplanar graphs}

\author{
Bo\v{s}tjan Bre\v{s}ar $^{a,b}$
\and
Nicolas Gastineau $^{c}$
\and
Olivier Togni $^{c}$
}

\date{\today}

\maketitle

\begin{center}
$^a$ Faculty of Natural Sciences and Mathematics, University of Maribor, Slovenia\\
{\tt bostjan.bresar@um.si}
\medskip

$^b$ Institute of Mathematics, Physics and Mechanics, Ljubljana, Slovenia\\
\medskip

$^c$ Laboratoire LIB, Université de Bourgogne Franche-Comté, France\\
{\tt Nicolas.Gastineau@u-bourgogne.fr} \\
{\tt olivier.togni@u-bourgogne.fr}
\medskip
\end{center}

\begin{abstract}
Given a graph $G$ and a nondecreasing sequence $S=(s_1,\ldots,s_k)$ of positive integers, the mapping $c:V(G)\longrightarrow \{1,\ldots,k\}$ is called an $S$-packing coloring of $G$ if for any two distinct vertices $x$ and $y$ in $c^{-1}(i)$, the distance between $x$ and $y$ is greater than $s_i$.  
The smallest integer $k$ such that there exists a $(1,2,\ldots,k)$-packing coloring of a graph $G$ is called the packing chromatic number of $G$, denoted $\pch(G)$. The question of boundedness of the packing chromatic number in the class of  subcubic (planar) graphs was investigated in several earlier papers; recently it was established that the invariant is unbounded in the class of all subcubic graphs.

In this paper, we prove that the packing chromatic number of any 2-connected bipartite subcubic outerplanar graph is bounded by $7$. Furthermore, we prove that every subcubic triangle-free outerplanar graph has a $(1,2,2,2)$-packing coloring, and that there exists a subcubic outerplanar graph with a triangle that does not admit a $(1,2,2,2)$-packing coloring. In addition, there exists a subcubic triangle-free outerplanar graph that does not admit a $(1,2,2,3)$-packing coloring.
A similar dichotomy is shown for bipartite outerplanar graphs: every such graph admits an $S$-packing coloring for $S=(1,3,\ldots,3)$, where $3$ appears $\Delta$ times ($\Delta$ being the maximum degree of vertices), and this property does not hold if one of the integers $3$ is replaced by $4$ in the sequence $S$. 
%This is due to the existence of a bipartite outer planar graph that does not admit an $(1,3,\ldots,3,4)$-packing coloring, $3$ appearing $\Delta-1$ times}.
\end{abstract}


\noindent
{\bf AMS Subj.\ Class.\ (2010)}: 05C15, 05C12, 05C70
\maketitle

\section{Introduction}
The $S$-packing chromatic number was introduced a decade ago in~\cite{goddard-2008} with  motivation coming from the frequency assignment problem. Roughly the idea of the concept is to generalize the classical coloring by involving the distance between vertices and allowing larger color values only for vertices that are more distant. Nevertheless, the problem has attracted the attention of many discrete mathematicians as it brings appealing combinatorial and computational challenges. 

Given a graph $G$ and a positive integer $d$, a set $A\subseteq V(G)$ is a {\em $d$-packing} in $G$ if for any two distinct vertices $x,y\in A$ the distance between $x$ and $y$ in $G$ is greater than $d$. For a nondecreasing sequence $S=(s_1,\ldots,s_k)$ of positive integers, the mapping $c:V(G)\longrightarrow \{1,\ldots,k\}$ is an {\em $S$-packing coloring} of $G$ if for every $i\in [k]$ the set $c^{-1}(i)$ is an $s_i$-packing. If there exists an $S$-packing coloring of $G$, we say that $G$ is {\em $S$-packing colorable}. If the sequence is $S=[k]$ for some positive integer $k$, we omit $S$ in the definition, and say that $G$ is {\em packing colorable} (as usual, we let $[k]=\{1,\ldots,k\}$). The smallest integer $k$ such that $G$ is packing colorable is {\em the packing chromatic number} of $G$, denoted $\pch(G)$. When we say that {\em the packing coloring condition} holds for a set $A\subseteq V(G)$ we mean that each set $A\cap c^{-1}(i)$, for all $i\ge 1$, is an $i$-packing in $G$. If $A\subset V(G)$, then by $G[A]$ we denote the subgraph of $G$ induced by $A$. 

A number of papers considered packing coloring of different infinite grids and lattices~\cite{bkr-2007, fiklli-2009, finbow-2010, kove-2014, SO2010}, where the most interesting development is about the infinite square grid; we mention only the recent paper on the topic~\cite{barnaby-2017} where it was shown that $13\le\pch(\mathbb{Z}\times\mathbb{Z})\le 15$, which is the latest refinement of the known bounds (initial bounds were presented already in the seminal paper~\cite{goddard-2008}). Fiala and Golovach have shown that the decision version of the packing chromatic number is NP-complete even in the class of trees~\cite{fiala-2010}. Packing coloring of some other classes of graphs, such as distance graphs~\cite{ekstein-2014, shao-2015,togni-2014}, hypercubes~\cite{torres-2015}, subdivision graphs of subcubic graphs~\cite{balogh-2018+,bkrw-2017b,gt-2016}, and some other classes of graphs~\cite{argiroffo-2014, jacobs-2013,lbe-2016} was also studied. 

One of the main questions in this area concerns graphs with bounded maximum degree $\Delta$, in particular, {\em subcubic graphs} (i.e., graphs with $\Delta=3$). For graphs with maximum degree $\Delta$, where $\Delta\ge 4$, the infinite $\Delta$-regular tree serves as an example showing that in this class of graphs the packing chromatic number is unbounded (in fact, Sloper proved this in the context of so-called eccentric colorings, but his result implies the same for the packing coloring~\cite{Slo}). On the other hand, the question whether in subcubic graphs the packing chromatic number is bounded was much more intriguing. It was posed in the seminal paper~\cite{goddard-2008}, and then investigated in several papers~\cite{bkr-2016,bkrw-2017a,gt-2016} using different approaches. Recently, Balogh, Kostochka and Liu~\cite{balogh-2018} have provided a negative answer to the question. Moreover, they proved that for every fixed $k$ and $g\ge 2k+2$, almost every $n$-vertex cubic graph of girth at least $g$ has the packing chromatic number greater than $k$. An explicit infinite family of subcubic graphs with unbounded packing chromatic number was then presented in~\cite{bf-2018}.

As the question was answered in the negative for all graphs with bounded maximum degree $3$, it becomes interesting for some subclasses of subcubic graphs. In particular, in~\cite{fiklli-2009} it was asked, whether there is an upper bound for the packing chromatic number of all planar cubic graphs, and this question was repeated in~\cite{bf-2018}. Very recently, the packing chromatic number of subcubic outerplanar graphs was considered~\cite{ght-2018+}. The upper bounds obtained in the paper involve the number of (internal) faces of the plane embedding of an outerplanar graph; for instance, it is proven that if $G$ is a 2-connected subcubic outerplanar graph with $r$ internal faces, then $\pch(G) \le 17\cdot 6^{3r}-2$. The question of boundedness of the packing chromatic number in subcubic outerplanar graphs thus seems widely open.  In this paper, we prove that, quite surprisingly, only 7 colors suffice if we restrict ourselves to the bipartite 2-connected case.

In the following section we fix the notation. In Section~\ref{sec:proof} we prove the following theorem, our main result. 

\begin{theo}\label{mainth}
Let $G$ be a $2$-connected bipartite subcubic outerplanar graph. Then $\chi_{\rho}(G)\le 7$.
\end{theo}

We continue in Sections~\ref{sec:133} and~\ref{sec:1222} with some results that are related to the coloring of the square of a graph. It was proven by Lih and Wang~\cite{lih-2006} that $\chi(G^2)\le \Delta(G)+2$ for an outerplanar graph $G$ (see also~\cite{agnarsson} for an extension), which confirms Wegner's old conjecture for planar graphs in the case of outerplanar graphs. In the language of $S$-packing colorings the result in~\cite{lih-2006} for outerplanar graphs $G$ with $\Delta(G)\le 3$ implies that $G$ is $(2,2,2,2,2)$-packing colorable. More generally, for an arbitrary $\Delta(G)\ge 3$, the result of Lih and Wang~\cite{lih-2006} gives the $(2,\ldots,2)$-packing colorability of outerplanar graphs $G$, where there are $\Delta(G)+2$ integers $2$. 

In Section~\ref{sec:133}, we present the following result about bipartite outerplanar graphs (i.e., no restriction to 2-connectedness and arbitrary maximum degree).
\begin{theo}
\label{th:133}
Let $G$ be a bipartite outerplanar graph. Let $S=(1,3,\ldots,3)$ be the sequence containing once the integer 1 and $k$ times the integer $3$, $k\ge 3$.
If $\Delta(G)\le k$, then $G$ is $S$-packing colorable.
\end{theo}
\noindent The result is complemented by an example showing that for $S=(1,3,\ldots,3,4)$, where $3$ appears $\Delta(G)-1$ times, there exists a bipartite outerplanar graph that does not admit an $S$-packing coloring. 

In Section~\ref{sec:1222}, subcubic outerplanar graphs are considered (extending the consideration of Theorem 1 to the non-bipartite case), and we prove:
\begin{theo}
\label{th:1222}
If $G$ is a subcubic outerplanar graph with no triangles, then $G$ is $(1,2,2,2)$-packing colorable.
\end{theo}
\noindent The result is complemented by an example showing that there exists a subcubic outerplanar graph (with triangles), which is not $(1,2,2,2)$-packing colorable. In addition, there exists a subcubic triangle-free outerplanar graph that does not admit a $(1,2,2,3)$-packing coloring.

In the final section, we present a variation of Theorem~\ref{th:1222} concerning the $(1,1,2)$-packing colorability of subcubic triangle-free outerplanar graphs, and pose some open problems.

\section{Notation}
\label{sec:notation}

A path between vertices $a$ and $b$ in a graph $G$ will be called an {\em $a,b$-path}. The length of a shortest $a,b$-path is the {\em distance} $d_G(a,b)$ between $a$ and $b$ in $G$ (we also write $d(a,b)$ if the graph is understood from the context). An $i$-packing in $G$ is a set of vertices $A$ such that for any two distinct vertices $x,y\in A$ we have $d_G(x,y)>i$. Clearly, a $1$-packing coincides with an independent set.

An outerplanar graph is a graph that has a planar drawing in which all vertices belong to the outer face of the drawing. Each time an outerplanar graph is considered, a drawing in which all vertices belong to the outer face of the drawing will be fixed. 
Let $G$ be an outerplanar graph. The {\em outer cycle} of $G$ corresponds to the cycle induced by the edges of the outer face.

For an outerplanar graph $G$, we denote by $\mathcal{T}_G$ the {\em weak dual} of $G$, i.e., the graph whose vertex set is the set of all inner faces of $G$, and $E(\mathcal{T}_G)=\{\alpha \beta |\ \alpha \text{ and  } \beta \text{ share a common edge} \}$.
For $\alpha\in V(\mathcal{T}_G)$, we denote by $C(\alpha)$, the (chordless) cycle in $G$ that corresponds to the face $\alpha$.  As any vertex $\alpha \in T_G$ corresponds to an inner face of $G$, we let $\alpha$ also denote this face and write $V(\alpha)$ for the set of its vertices.

Let $G$ be a 2-connected outerplanar graph (we consider such graphs in Section~\ref{sec:proof}). Note that in this case $\mathcal{T}_G$ is a tree. 
%Let $\omega_0$ be a vertex of $\mathcal{T}_G$ of minimum eccentricity. 
We consider $\mathcal{T}_G$ as a rooted tree with an arbitrary chosen vertex $\omega_0$ as the root. The notions of  {\em parent, child} and {\em descendant} should be clear in this context.  We can also define the {\em depth} of a vertex $\beta \in V (\mathcal{T}_G)$, denoted by $p(\beta)$, as $d_{\mathcal{T}_G}(\beta,\omega_0)$. The depth of $\mathcal{T}_G$, denoted by $p(\mathcal{T}_G)$, is the maximum value of $p(\beta)$, for $\beta\in V(\mathcal{T}_G)$.

\section{Packing coloring of $2$-connected bipartite subcubic outerplanar graphs}
\label{sec:proof}

In this section we prove that a $2$-connected bipartite subcubic outerplanar graph $G$ has packing chromatic number bounded by $7$, i.e., $\chi_{\rho}(G)\le 7$. At the end of the section we add a result (Proposition 5) which shows that this is best possible.

The proof of the main theorem has two steps. In the first step we construct a subset $B$ of $V(G)$, and present a coloring $f$ of the vertices of $A=V(G)\setminus B$ by using only the colors from $\{1,2,3\}$ such that the packing coloring condition holds for $A$. The set $B$ will be called the set of {\em big vertices}, and so any vertex in $B$ is called a {\em big vertex}. The big vertices will be colored in the second step by using only the colors from $\{4,5,6,7\}$.  That is, we will extend $f$ from $A$ to all vertices of $V(G)$. We will prove that $f:V(G)\longrightarrow \{1,\ldots,7\}$ has four special properties, which will be helpful in proving that $f$ is a packing coloring of $G$ (that is, for any two distinct vertices $u,v\in V(G)$, $f(u)=f(v)$ implies $d_G(u,v)>f(u)$). As in many cases, additional (technical) conditions, which are not needed in the result, are helpful in the proof.  

\medskip

\noindent {\bf Proof of Theorem~\ref{mainth}.}

Let $G$ be a $2$-connected bipartite subcubic outerplanar graph. Let $\omega_0$ be a vertex of $\mathcal{T}_G$ of minimum eccentricity; we consider $\mathcal{T}_G$ to be a rooted tree with $\omega_0$ as its root. 

In the proof we will construct a set of vertices $B$ and a packing coloring $f$ of $G$.
% (The vertices of $B$ will be called big vertices, and their colors, determined in Step 2 of the proof, will be from the set $\{4,5,6,7\}$. The colors of vertices in $A=V(G)\setminus B$, determined in Step 1 of the proof, will be from the set $\{1,2,3\}$.) 
The coloring $f$ will satisfy the following additional properties.
\begin{enumerate}[(i)]
\item Any vertex with color different from $1$ (including big vertices whose color will be determined in Step 2) has all its neighbors colored by color 1.
\item Any face $\alpha$ of $G$ contains exactly one big vertex if $|\alpha|\ge 6$ and at most one big vertex if $|\alpha|=4$.
\item Any big vertex is at distance at least $4$ from any other big vertex.
\item Any vertex with color from $\{6,7\}$ is at distance at least $6$ from any vertex with color from $\{6,7\}$.

\end{enumerate}

{\bf Step 1.} In this first step of the proof, we construct the set $B$ and consequently the set $A=V(G)\setminus B$, and color the vertices of $A$ by using only the colors from $\{1,2,3\}$.
% (Note that $B=V(G)\setminus A$ is a set that will contains only big vertices.) 
During this step, we will ensure that properties (i), (ii) and (iii) are satisfied (while property (iv) will be verified in the second step, when we assign colors to the vertices from $B$).

The proof uses the structure of the tree $\mathcal{T}_G$. We consider the faces in a Breadth-first search (BFS) order by starting with the face $\omega_0$. In each facial cycle we will repeatedly use the pattern $1,2,1,3$; by this we mean that vertices along the cycle will follow in the order $1,2,1,3,\ldots$ or $1,3,1,2,\ldots$, which also applies when the length of the cycle is not divisible by $4$ (in which case, we omit the appropriate number of colors at the end of the sequence). 
If the length of the corresponding cycle is greater than $4$, one of the vertices is taken as a big vertex and is not yet colored (it will not belong to the set $A$),  and for the rest of the cycle we use the pattern $1,2,1,3$ (by repeating it an appropriate number of times). In a 4-cycle we may either use only the colors from $\{1,2,3\}$ or declare one vertex as big,  which depends on the type of the used coloring (described soon).

Clearly, all the vertices of $\omega_0$ can be colored by using the above pattern. 
In particular, if $p(\mathcal{T}_G)=0$, then $G$ is the cycle $C(\omega_0)$, so that we can set $A=V(\omega_0)$, and the described coloring satisfies properties (i), (ii), (iii) and (iv).

Suppose now $p(\mathcal{T}_G)>0$. By following a BFS order of $\mathcal{T}_G$, consider a face $\alpha\in \mathcal{T}_G$, $\alpha\neq \omega_0$, where $\omega \in \mathcal{T}_G$ is the parent of $\alpha$. By the construction, a big vertex $u$ of $\omega$ is already determined (including the possibility that $\omega$ has no big vertices, which may happen when the size of $C(\omega)$ is $4$), and other vertices of $\omega$ are colored by colors $\{1,2,3\}$ repeatedly using the pattern $1,2,1,3$.

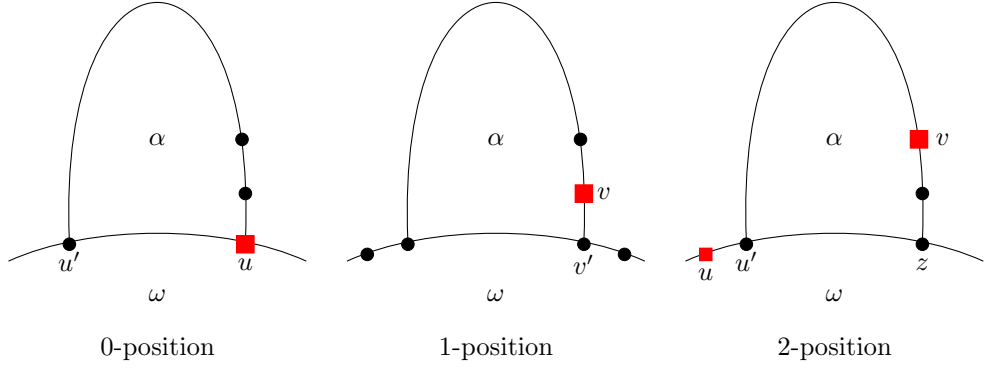
\begin{figure}[t]

\begin{center}

\begin{tikzpicture}[scale=0.9]
\draw  (-2.2,-0.2) .. controls (-1,0.35) and (1,0.35) .. (2.2,-0.2);
\draw  (-1.3,0.1) .. controls (-1.5,4.8) and (1.5,4.8) .. (1.3,0.1);
\node at (-1.3,0.05) [circle,draw=black,fill=black,scale=0.5]{};
\node at (1.3,0.05) [regular polygon, regular polygon sides=4,draw=red,fill=red,scale=0.7]{};
\node at (1.3,0.8) [circle,draw=black,fill=black,scale=0.5]{};
\node at (1.25,1.6) [circle,draw=black,fill=black,scale=0.5]{};
%\node at (-1.3,0.8) [circle,draw=black,fill=black,scale=0.5]{};
\node at (-1.3,-0.2) {$u'$};
\node at (1.3,-0.25) {$u$};
\node at (0,-0.7) {$\omega$};
\node at (0,-1.5){0-position};
\node at (0,1.6) {$\alpha$};
%\node at (0,3.9) {$\ldots$};

\draw  (-2.2+5,-0.2) .. controls (-1+5,0.35) and (1+5,0.35) .. (2.2+5,-0.2);
\draw  (-1.3+5,0.1) .. controls (-1.5+5,4.8) and (1.5+5,4.8) .. (1.3+5,0.1);
\node at (-1.3+5,0.05) [circle,draw=black,fill=black,scale=0.5]{};
\node at (1.3+5,0.05) [circle,draw=black,fill=black,scale=0.5]{};
\node at (1.25+5,1.6) [circle,draw=black,fill=black,scale=0.5]{};
\node at (1.3+5,0.8) [regular polygon, regular polygon sides=4,draw=red,fill=red,scale=0.7]{};
\node at (0+5,-0.7) {$\omega$};
\node at (0+5,-1.5){1-position};
\node at (0+5,1.6) {$\alpha$};
%\node at (0+5,3.9) {$\ldots$};
\node at (1.3+5,-0.25) {$v'$};
\node at (1.6+5,0.8) {$v$};
\node at (-1.9+5,-0.1) [circle,draw=black,fill=black,scale=0.5]{};
\node at (1.9+5,-0.1) [circle,draw=black,fill=black,scale=0.5]{};

\draw  (-2.2+10,-0.2) .. controls (-1+10,0.35) and (1+10,0.35) .. (2.2+10,-0.2);
\draw  (-1.3+10,0.1) .. controls (-1.5+10,4.8) and (1.5+10,4.8) .. (1.3+10,0.1);
\node at (-1.3+10,0.05) [circle,draw=black,fill=black,scale=0.5]{};
\node at (1.3+10,0.05) [circle,draw=black,fill=black,scale=0.5]{};
\node at (1.3+10,0.8) [circle,draw=black,fill=black,scale=0.5]{};
\node at (1.25+10,1.6) [regular polygon, regular polygon sides=4,draw=red,fill=red,scale=0.7]{};
\node at (-1.9+10,-0.1) [regular polygon, regular polygon sides=4,draw=red,fill=red,scale=0.5]{};
\node at (10,-0.7) {$\omega$};
\node at (10,-1.5){2-position};
\node at (10,1.6) {$\alpha$};
%\node at (10,3.9) {$\ldots$};
\node at (-1.3+10,-0.2) {$u'$};
\node at (-1.9+10,-0.4) {$u$};
\node at (1.3+10,-0.25) {$z$};
\node at (1.6+10,1.6) {$v$};

\end{tikzpicture}
\end{center}
\caption{0- 1- and 2-position in face $\alpha$ (circle: vertex of $A$; square: vertex of $B$, i.e. big vertex).}
\label{fig012}
\end{figure}

We consider the following three cases, depending on the position of the big vertex in $\alpha$ (and the coloring of the vertices of $A$), respectively called $0$-, $1$- and $2$-position, and extend the function $f$ to all vertices of $\alpha\cap A$ in each of these cases, referred to as a $0$-, $1$- and $2$-coloring, respectively.

\begin{itemize}

\item {\bf 0-position}: suppose that the big vertex $u$ of $\omega$ coincides with one of the two vertices belonging to $\alpha\cap \omega$. In this case, the big vertex of $\alpha$ is also determined, notably, $u\in B\cap \alpha$. Let $u'$ be the other vertex belonging to $\alpha\cap \omega$. (See Figure~\ref{fig012}.) Thanks to property (i), $u'$ is already colored by 1. The corresponding {\bf $0$-coloring} of the vertices of $C(\alpha)$ is obtained by starting with the uncolored neighbor of $u$, and repeatedly using the pattern $1,2,1,3$ along the cycle (or, as mentioned before when describing the setting with the pattern $1,2,1,3$, the actual pattern may also be $1,3,1,2$ depending on the color of the neighbor of $u'$ belonging to $\omega\setminus\alpha$), ensuring that the neighbors of $u'$ get different colors; note that their colors will be $k$, $2$ and $3$, respectively, $k$ being the color that $u$ will get in step 2.

\item {\bf 1-position}: suppose that the big vertex of $\omega$ is not one of the vertices belonging to $\alpha\cap \omega$ and, in addition, not adjacent to any of the two vertices belonging to $\alpha\cap \omega$. (In other words, the four vertices of $C(\omega)$ that are closest to $C(\alpha)$ are colored by colors $1,2,1$ and $3$.) Let $v'$ be the vertex belonging to $\alpha\cap \omega$, for which $f(v')=1$. And thus other big vertices in $\omega$ are at distance at least three from $v'$ by (i). In this case, we define the big vertex $v$ of $\alpha$ (i.e. the vertex $v$ of $\alpha$ belonging to $B$) as the neighbor of $v'$, which does not lie in $\omega$. (See Figure~\ref{fig012}.) The corresponding {\bf $1$-coloring} of the vertices of $C(\alpha)$ is obtained by starting with the neighbor of $v'$ belonging to $\alpha\cap \omega$, and repeatedly using the pattern $1,2,1,3$ along the cycle.

\item {\bf 2-position}: suppose that the big vertex $u$ of $\omega$ is not one of the vertices belonging to $\alpha\cap \omega$, but is adjacent to one of the two vertices belonging to $\alpha\cap \omega$. Let us denote by $u'$ the neighbor of $u$ belonging to $\alpha\cap \omega$, and let the other vertex belonging to $\alpha\cap \omega$ be called $z$. By (i), $f(u')=1$, and $f(z)\in \{2,3\}$. 
Now, if $C(\alpha)$ has 4 vertices, then we use only colors $1,2$ or $1,3$ ensuring that the neighbors of $u'$ in $C(\alpha)$ get distinct colors. Otherwise, we define the big vertex $v$ of $\alpha$ (i.e. the vertex $v$ of $\alpha$ belonging to $B$) to be the vertex at distance $2$ from the vertex $z$, different from a neighbor of $u'$. (See Figure~\ref{fig012}.)  The corresponding {\bf $2$-coloring} of the vertices of $C(\alpha)$ is obtained by starting with $u'$, and repeatedly using the pattern $1,2,1,3$  along the cycle, ensuring that the neighbors of $u'$ get different colors in $\{2,3\}$. 
\end{itemize}

Clearly, for any of the three colorings ($0$-coloring, $1$-coloring, $2$-coloring) the packing coloring condition holds for the vertices of $A$ belonging to a face $\alpha$ following $\omega$ in the BFS order of $\mathcal{T}_G$, and properties (i), (ii) and (iii) extend from vertices colored so far to the vertices of $\alpha$.  We derive the following observation.

\begin{lem}
Let $B$ be the set constructed in Step 1 with $A=V(G)\setminus B$, and let $f$ be the coloring of the vertices of $A$ by colors $\{1,2,3\}$ as described above. Then $f$ and $B$ satisfy properties (i), (ii) and (iii), and for any two distinct vertices $x$ and $y$ in $A$ such that $f(x)=f(y)$ we have $d_G(x,y)>f(x)$. 
\end{lem}

{\bf Step 2.} In this step we need to determine the $f$-values of big vertices, and prove that property (iv) holds for these vertices and that the packing coloring condition holds for the set $B$. We determine the colors of big vertices following a BFS order on $\mathcal{T}_G$; we start by determining the possible color of the big vertex of $\omega_0$. If such a vertex exists (i.e., $|V(\omega_0)|>4$), then we color it by $4$.

A big vertex $x$ that belongs to a face $\beta$ will be called a big vertex {\em arising} from $\alpha$ if the following two conditions are true:
\begin{enumerate}
\item[i)] $\beta$ is a descendant of $\alpha$ with respect to $\mathcal{T}_G$;
\item[ii)] $x$ is at distance at most two from $C(\alpha)$.
\end{enumerate}
A step of the coloring construction consists in dealing with a face $\alpha$ that comes next in the chosen BFS order (beginning with the face $\omega_0$), and coloring all the big vertices arising from $\alpha$.
Let $\omega$ be the parent of the face $\alpha$ (if it exists).
In a step of the coloring construction, it is supposed that, if $\alpha\neq \omega_0$, then both the big vertex of $\omega$ and the big vertices arising from $\omega$ are already colored.
Consequently, the big vertex of $\alpha$ is already colored, because it arises from $\omega$.
On the other hand, when $\alpha=\omega_0$, only the big vertex of $\omega_0$ is already colored.

We assume that $f$ satisfies the packing coloring condition for big vertices colored in previous steps, and that property (iv) is satisfied for the big vertices colored in previous steps. Consequently, we have to prove that the packing coloring condition holds for big vertices and that property (iv) remains true when we extend the function $f$ in a new step.
We distinguish three cases with respect to the type of the position of the big vertex ($0$-position, $1$-position, $2$-position) and the corresponding coloring, used to color the big vertex of $\alpha$. We will call the colors in $\{6,7\}$  {\em very big}.

\bigskip

\textbf{Case 1.} $\alpha$ is in 0-position, see Figure~\ref{fig:0pos}.

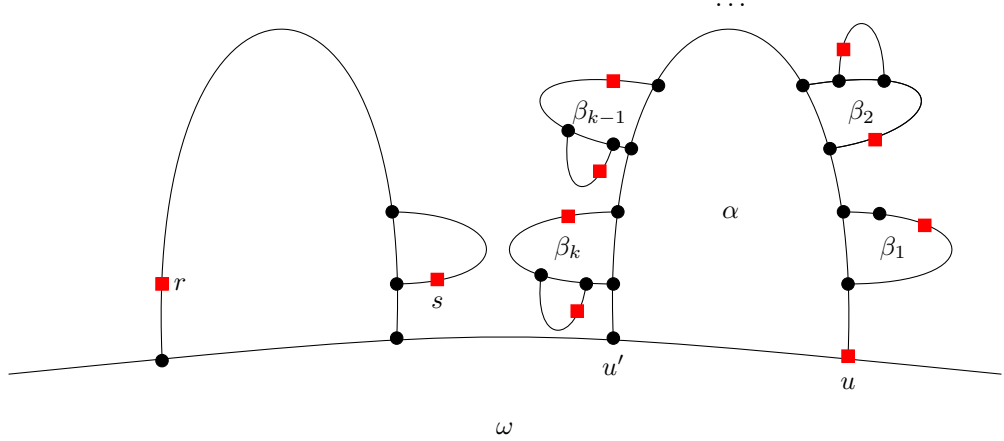
\begin{figure}[t]

\begin{center}

\begin{tikzpicture}[scale=1.2]
\draw  (-3,-0.2) .. controls (2.5,0.35) .. (8,-0.2);
\draw  (-1.3,0) .. controls (-1.5,4.8) and (1.5,4.8) .. (1.3,0.2);
\draw  (3.7,0.2) .. controls (3.5,4.8) and (6.5,4.8) .. (6.3,0);
\draw  (-1.3+5,0.8) .. controls (-2.9+5,0.8) and (-2.8+5,1.6) .. (-1.2+5,1.6);

\draw  (-1.1+5,2.3) .. controls (-2.4+5,2.5) and (-2.6+5,3.3) .. (-0.8+5,3);
\draw  (1.3+5,0.8) .. controls (2.9+5,0.8) and (2.8+5,1.6) .. (1.2+5,1.6);

\draw  (1.25,0.8) .. controls (2.7,0.8) and (2.6,1.6) .. (1.2,1.6);
\node at (1.75,0.85) [regular polygon, regular polygon sides=4,draw=red,fill=red,scale=0.5]{};

\draw  (1.1+5,2.3) .. controls (2.4+5,2.5) and (2.6+5,3.3) .. (0.8+5,3);
\draw  (1.1+5,2.3) .. controls (2.4+5,2.5) and (2.6+5,3.3) .. (0.8+5,3);

\draw  (-1.8+5,2.5) .. controls (-1.9+5,1.7) and (-1.5+5,1.7) .. (-1.3+5,2.35);
\node at (-1.8+5,2.5) [circle,draw=black,fill=black,scale=0.5]{};
\node at (-1.3+5,2.35) [circle,draw=black,fill=black,scale=0.5]{};
\node at (-1.45+5,2.05) [regular polygon, regular polygon sides=4,draw=red,fill=red,scale=0.5]{};

\draw  (1.2+5,3.05) .. controls (1.2+5,3.9) and (1.7+5,3.9) .. (1.7+5,3.05);
\node at (1.2+5,3.05) [circle,draw=black,fill=black,scale=0.5]{};
\node at (1.7+5,3.05) [circle,draw=black,fill=black,scale=0.5]{};
\node at (1.25+5,3.4) [regular polygon, regular polygon sides=4,draw=red,fill=red,scale=0.5]{};

\draw  (-2.1+5,0.9) .. controls (-2.2+5,0.1) and (-1.7+5,0.1) .. (-1.6+5,0.8);
\node at (-2.1+5,0.9)  [circle,draw=black,fill=black,scale=0.5]{};
\node at (-1.6+5,0.8) [circle,draw=black,fill=black,scale=0.5]{};
\node at (-1.7+5,0.5) [regular polygon, regular polygon sides=4,draw=red,fill=red,scale=0.5]{};

\node at (-1.3+5,0.2) [circle,draw=black,fill=black,scale=0.5]{};
\node at (1.3+5,0) [regular polygon, regular polygon sides=4,draw=red,fill=red,scale=0.5]{};
\node at (-1.25+5,1.6) [circle,draw=black,fill=black,scale=0.5]{};
\node at (-1.1+5,2.3) [circle,draw=black,fill=black,scale=0.5]{};
\node at (-0.8+5,3) [circle,draw=black,fill=black,scale=0.5]{};

\node at (1.3+5,0.8) [circle,draw=black,fill=black,scale=0.5]{};
\node at (1.25+5,1.6) [circle,draw=black,fill=black,scale=0.5]{};
\node at (1.1+5,2.3) [circle,draw=black,fill=black,scale=0.5]{};
\node at (0.8+5,3) [circle,draw=black,fill=black,scale=0.5]{};

\node at (1.65+5,1.58) [circle,draw=black,fill=black,scale=0.5]{};
\node at (2.15+5,1.45) [regular polygon, regular polygon sides=4,draw=red,fill=red,scale=0.5]{};
\node at (1.6+5,2.4) [regular polygon, regular polygon sides=4,draw=red,fill=red,scale=0.5]{};

\node at (-1.3+5,0.8) [circle,draw=black,fill=black,scale=0.5]{};
\node at (-1.3+5,3.05) [regular polygon, regular polygon sides=4,draw=red,fill=red,scale=0.5]{};

\node at (-1.8+5,1.55) [regular polygon, regular polygon sides=4,draw=red,fill=red,scale=0.5]{};

\node at (7.3+5,-0.12) [circle,draw=black,fill=black,scale=0.5]{};

\node at (-1.3,-0.05) [circle,draw=black,fill=black,scale=0.5]{};
\node at (1.3,0.2) [circle,draw=black,fill=black,scale=0.5]{};
\node at (1.3,0.8) [circle,draw=black,fill=black,scale=0.5]{};
\node at (1.25,1.6) [circle,draw=black,fill=black,scale=0.5]{};

\node at (-1.3,0.8) [regular polygon, regular polygon sides=4,draw=red,fill=red,scale=0.5]{};

\node at (-1.3+5,-0.1) {$u'$};
\node at (1.3+5,-0.3) {$u$};
\node at (2.5,-0.8) {$\omega$};
\node at (0+5,1.6) {$\alpha$};
\node at (1.8+5,1.2) {$\beta_{1}$};
\node at (1.45+5,2.7) {$\beta_{2}$};
\node at (-1.8+5,1.2) {$\beta_{k}$};
\node at (-1.45+5,2.7){$\beta_{k-1}$};
\node at (0+5,3.9) {$\ldots$};
\node at (-1.1,0.8) {$r$};
\node at (1.75,0.6) {$s$};

\end{tikzpicture}
\end{center}
\caption{Big vertices arising from face $\alpha$ in 0-position (circle: vertex of $A$; square: vertex of $B$, i.e. big vertex).}
\label{fig:0pos}
\end{figure}

Consider the face $\omega$ and its descendants with respect to $\mathcal{T}_G$, and note that the color of their big vertices could already have been determined (with the exception of big vertices arising from $\alpha$). By property (i), all big vertices are at even distance from vertex $u$, and they are at distance at least 4 from $u$ by property (iii). 
We distinguish two kinds of big vertices that were already colored, namely those that are at a shorter distance from $u$ than from $u'$, and those that are at a shorter distance from $u'$ than from $u$. Those that are closer to $u$ than to $u'$ are at distance at least $8$ from big vertices that arise from $\alpha$. (As shown in Figure~\ref{fig:0pos}, big vertices arising from $\alpha$ belong to faces $\beta_1,\ldots,\beta_k$, which are children of $\alpha$ with respect to $\mathcal{T}_G$, or to their children.) On the other hand, there can be big vertices, which are at distance $3$ from $u'$ (and $4$ from $u$), and have already been colored. More precisely, by using property (iii), and the fact that the neighbor of $u'$ that is not in $\alpha$ has at most two other neighbors, we find that there can be at most two such big vertices, which are at distance $3$ from $u'$. If they indeed exists, we denote them by $r$ and $s$, and note that $r,s$ and $u$ are pairwise at distance $4$ from each other. 

Next we analyze possible positions of vertices that arise from $\alpha$. For every vertex $a$ in $C(\alpha)$, which is at an even distance from $u$, there can be a big vertex $a'$ that arises from $\alpha$ such that there is an $a,a'$-path of length $2$ outside $\alpha$ (there is at most one such vertex by property (iii)). On the other hand, for every vertex $b$ in $C(\alpha)$, which is at an odd distance from $u$, there can be a big vertex $b'$ that arises from $\alpha$, which is adjacent to $b$ (with the exception of the neighbors of $u$ in $C(\alpha)$, which cannot be adjacent to another big vertex due to property (iii)). The situation when all these big vertices exist is described in Figure~\ref{fig:case1}.

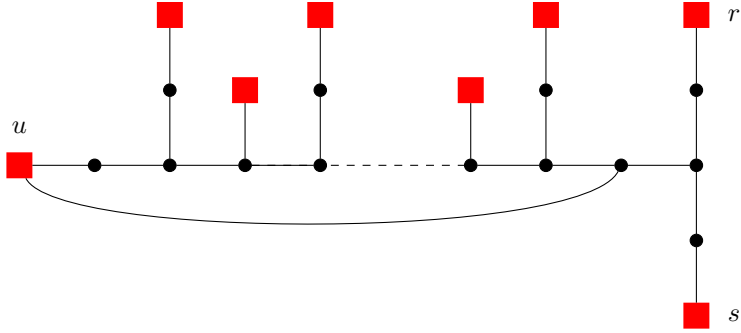
\begin{figure}[ht]
\begin{center}
\begin{tikzpicture}[scale=1]
\node at (0,0.5) {$u$};
\node at (0,0) [regular polygon, regular polygon sides=4,draw=red,fill=red,scale=1](u){};
\node at (1,0) [circle,draw=black,fill=black,scale=0.5](u1){};
\node at (2,0) [circle,draw=black,fill=black,scale=0.5](u2){};
\node at (2,1) [circle,draw=black,fill=black,scale=0.5](u3){};
\node at (2,2) [regular polygon, regular polygon sides=4,draw=red,fill=red,scale=1](u4){};
\node at (3,0) [circle,draw=black,fill=black,scale=0.5](u5){};
\node at (3,1) [regular polygon, regular polygon sides=4,draw=red,fill=red,scale=1](u6){};
\node at (4,0) [circle,draw=black,fill=black,scale=0.5](u7){};
\node at (4,1) [circle,draw=black,fill=black,scale=0.5](u8){};
\node at (4,2) [regular polygon, regular polygon sides=4,draw=red,fill=red,scale=1](u9){};
\node at (6,0) [circle,draw=black,fill=black,scale=0.5](u10){};
\node at (6,1) [regular polygon, regular polygon sides=4,draw=red,fill=red,scale=1](u11){};
\node at (7,0) [circle,draw=black,fill=black,scale=0.5](u12){};
\node at (7,1) [circle,draw=black,fill=black,scale=0.5](u13){};
\node at (7,2) [regular polygon, regular polygon sides=4,draw=red,fill=red,scale=1](u14){};
\node at (8,0) [circle,draw=black,fill=black,scale=0.5](u15){};
\node at (9,0) [circle,draw=black,fill=black,scale=0.5](u16){};
\node at (9,1) [circle,draw=black,fill=black,scale=0.5](u17){};
\node at (9,-1) [circle,draw=black,fill=black,scale=0.5](u18){};
\node at (9,2) [regular polygon, regular polygon sides=4,draw=red,fill=red,scale=1](u20){};
\node at (9,-2) [regular polygon, regular polygon sides=4,draw=red,fill=red,scale=1](u21){};
\node at (9.5, 2) {$r$};
\node at (9.5, -2) {$s$};
\draw (u) -- (u1) -- (u2) -- (u3) -- (u4);
\draw (u2) -- (u5) -- (u6);
\draw (u5) -- (u7) -- (u8) -- (u9);
\draw [style=dashed] (u5) -- (u10);
\draw (u11) -- (u10) -- (u12) -- (u13) -- (u14);
\draw (u12) -- (u15) .. controls (7.5,-1) and (.5,-1) .. (u);
\draw (u15) -- (u16) -- (u17) -- (u20);
\draw (u16) -- (u18) -- (u21);

\end{tikzpicture}
\end{center}
\caption{\label{fig:case1} Case 1 ($\alpha$ in $0$-position).}
\end{figure}

In the most complex case when both vertices $r$ and $s$ exist and are already colored, vertices $u$, $r$ and $s$ are pairwise at distance $4$ from each other. By property (iv), exactly one of these three vertices is colored by a very big color. 
We consider two subcases, depending on the color of vertices $u$, $r$ and $s$. 
Note that we can easily modify our coloring function so that all vertices colored by color $4$ are colored by color $5$ and vice-versa. Also, the same  holds for the vertices colored by color $6$ and $7$ (we can exchange the two color classes).
A consequence is that the way to deal with the case $f(u)=5$, $\{f(r),f(s)\}=\{4,7\}$ is the same as when $f(u)=4$, $\{f(r),f(s)\}=\{5,7\}$ or when $f(u)=5$, $\{f(r),f(s)\}=\{4,6\}$ and that the way to deal with the case $f(u)=7$, $\{f(r),f(s)\}=\{4,5\}$ is the same as when $f(u)=6$, $\{f(r),f(s)\}=\{4,5\}$.
%We can assume, without loss of generality, that if the color of $u$ is very big, then it is $6$ and otherwise it is $4$; we can also assume without loss of generality that in the case $f(u)=4$, we have $\{f(r),f(s)\}=\{5,7\}$. 
Thus, it suffices to deal only with two subcases: $f(u)=4$, $\{f(r),f(s)\}=\{5,7\}$ and  $f(u)=6$, $\{f(r),f(s)\}=\{4,5\}$.
%This is described in the next table.

%\begin{center}
%\begin{tabular}{c|c|c}
%Subcase & a & b\\\hline
%$f(u)$ & 4 & 6\\
%$\{f(r),f(s)\}$ & $\{5,7\}$ & $\{4,5\}$\\
%\end{tabular}
%\end{center}

\medskip

\textbf{Subcase 1.a.} $f(u)=4$ and $\{f(r),f(s)\}=\{5,7\}$. 
We present three patterns that define the coloring of the big vertices with respect to different lengths $n$ of the cycle $C(\alpha)$. The patterns give $f$-values of the big vertices following their presentation in Figure~\ref{fig:case1} (from left to right); in the case when some of the big vertices that arise from $\alpha$ do not exist, we simply skip the corresponding values in the pattern. 

If $n=4$, there is at most one big vertex that arises from $\alpha$, and we color it by $5$. Now, assume that $n\ge6$. Note that the numbers between vertical bars are to be repeated $k$ times (case $k=0$ included).
\medskip

Pattern for length $n=4k+8$:
\begin{tabular}{cc|cccc|ccc}
%\hline
7 & & 6 & & 7 & & 4 & & 5\\
& 5 & & 4 && 5 && 6 &\\%\hline
\end{tabular}
\bigskip

Pattern for length $n=8k+6$:
\begin{tabular}{cc|cccccccc|c}
%\hline
5 && 5 && 5 && 5 && 5 && 5\\
& 6 && 4 && 7 && 4 && 6 &\\%\hline
\end{tabular}

\bigskip
Pattern for length $n=8k+10$:

\medskip

\begin{tabular}{cccccc|cccccccc|c}
%\hline
5 && 5 && 5 && 5 && 5 && 5 && 5 && 5\\
& 7 && 4 && 6 && 4 && 7 && 4 && 6 &\\%\hline
\end{tabular}

\medskip

Note that because $f(u)=4$, the first and the last two values in the patterns cannot be  $4$, and because $\{f(r),f(s)\}=\{5,7\}$, the last two values in the patterns cannot be $7$. Also note that two vertices that correspond to two successive values in the upper row of the pattern are at distance 6, two vertices that correspond to two successive values in the lower row of the pattern are at distance 4, and two vertices that correspond to two consecutive values in the patterns (one in the upper and the other in the lower row) are at distance 4. 

In each of the patterns, one can check that property (iv) holds, and that for any two identical numbers, the distance between the corresponding vertices in $G$ is bigger than this number (i.e., the packing coloring condition is satisfied).
Note that for $k=0$ one needs to omit the numbers that are between the vertical bars in each pattern, which covers the lengths of cycles $n$, where $n\in \{6,8,10\}$.

\textbf{Subcase 1.b.} $f(u)=6$ and $\{f(r),f(s)\}=\{4,5\}$. If $n=4$ we can use color $4$ for the only big vertex that possibly arises from $\alpha$. Now, assume that $n\ge6$.
We present patterns that define the coloring of the big vertices in the case $f(u)=6$, and $\{f(r),f(s)\}=\{4,5\}$. Note that it implies that the first and the last four values in the patterns cannot be $6$. Again the numbers between vertical bars are to be repeated $k$ times (including $k=0$, where we omit the numbers between vertical bars):

\bigskip
Pattern for $n=4k+8$:
\begin{tabular}{cc|cccc|ccc}
%\hline
5 & & 7 & & 6 & & 7 & & 4\\
& 4 & & 5 && 4 && 5 &\\%\hline
\end{tabular}

\bigskip
Pattern for $n=4k+10$:
\begin{tabular}{ccccc|cccc|cc}
%\hline
5 && 5 && 4 && 6 && 7 && 4\\
& 4 && 7 && 5 && 4 && 5 &\\%\hline
\end{tabular}

\medskip 

Pattern for $n=6$:
\begin{tabular}{ccc}
%\hline
4 & & 4\\
& 5 &\\%\hline
\end{tabular}

\medskip 

\bigskip 

\textbf{Case 2.} $\alpha$ is in 1-position, see Figure~\ref{fig:1pos}.

\begin{figure}[t]

\begin{center}

\begin{tikzpicture}[scale=1.2]
\draw  (-3,-0.2) .. controls (2.5,0.35) .. (8,-0.2);
\draw  (-1.3,0) .. controls (-1.5,4.8) and (1.5,4.8) .. (1.3,0.2);
\draw  (3.7,0.2) .. controls (3.9,4.8) and (6.5,4.8) .. (6.3,0);
\draw  (-1.3,0.8) .. controls (-2.9,0.8) and (-2.8,1.6) .. (-1.2,1.6);
\draw  (-1.1,2.3) .. controls (-2.4,2.5) and (-2.6,3.3) .. (-0.8,3);
\node at (-1.3,0) [circle,draw=black,fill=black,scale=0.5]{};
\node at (1.3,0.2) [circle,draw=black,fill=black,scale=0.5]{};
\node at  (3.7,0.2) [circle,draw=black,fill=black,scale=0.5]{};
\node at  (6.3,0) [circle,draw=black,fill=black,scale=0.5]{};
\node at (-1.25,1.6) [circle,draw=black,fill=black,scale=0.5]{};
\node at (-1.1,2.3) [circle,draw=black,fill=black,scale=0.5]{};
\node at (-0.8,3) [circle,draw=black,fill=black,scale=0.5]{};
\draw  (1.3,0.8) .. controls (2.9,0.8) and (2.8,1.6) .. (1.2,1.6);
\draw  (1.1,2.3) .. controls (2.4,2.5) and (2.6,3.3) .. (0.8,3);
\node at (1.3,0.8) [circle,draw=black,fill=black,scale=0.5]{};
\node at (1.25,1.6) [circle,draw=black,fill=black,scale=0.5]{};
\node at (1.1,2.3) [circle,draw=black,fill=black,scale=0.5]{};
\node at (0.8,3) [circle,draw=black,fill=black,scale=0.5]{};

\draw  (1.1,2.3) .. controls (2.4,2.5) and (2.6,3.3) .. (0.8,3);

\draw  (-1.8,2.5) .. controls (-1.9,1.7) and (-1.5,1.7) .. (-1.3,2.35);
\node at (-1.8,2.5) [circle,draw=black,fill=black,scale=0.5]{};
\node at (-1.3,2.35) [circle,draw=black,fill=black,scale=0.5]{};
\node at (-1.45,2.05) [regular polygon, regular polygon sides=4,draw=red,fill=red,scale=0.5]{};

\draw  (1.2,3.05) .. controls (1.2,3.9) and (1.7,3.9) .. (1.7,3.05);
\node at (1.2,3.05) [circle,draw=black,fill=black,scale=0.5]{};
\node at (1.7,3.05) [circle,draw=black,fill=black,scale=0.5]{};
\node at (1.25,3.4) [regular polygon, regular polygon sides=4,draw=red,fill=red,scale=0.5]{};

\draw  (1.6,1.58) .. controls (1.6,2.4) and (2.1,2.4) .. (2.1,1.48);
\node at (1.6,1.58) [circle,draw=black,fill=black,scale=0.5]{};
\node at (2.1,1.48) [circle,draw=black,fill=black,scale=0.5]{};
\node at (1.65,1.9) [regular polygon, regular polygon sides=4,draw=red,fill=red,scale=0.5]{};

\node at (-1.3,0.8) [regular polygon, regular polygon sides=4,draw=red,fill=red,scale=0.5]{};
\node at (1.8,0.85) [regular polygon, regular polygon sides=4,draw=red,fill=red,scale=0.5]{};
\node at (1.6,2.4) [regular polygon, regular polygon sides=4,draw=red,fill=red,scale=0.5]{};
\node at (-1.3,3.05) [regular polygon, regular polygon sides=4,draw=red,fill=red,scale=0.5]{};

\node at (-2.3,-0.12) [circle,draw=black,fill=black,scale=0.5]{};

\node at (3.75,0.8) [regular polygon, regular polygon sides=4,draw=red,fill=red,scale=0.5]{};

\node at (-1.3,-0.3) {$v'$};
\node at (-1.1,0.8) {$v$};
\node at (1.8,0.6) {$p$};
\node at (3.55,0.8) {$r$};
\node at (2.5,-0.8) {$\omega$};
\node at (0,1.6) {$\alpha$};
\node at (1.8,1.2) {$\beta_{1}$};
\node at (1.45,2.7) {$\beta_{2}$};
\node at (-1.8,1.2) {$\beta_{k}$};
\node at (-1.45,2.7){$\beta_{k-1}$};
\node at (0,3.9) {$\ldots$};

\end{tikzpicture}
\end{center}
\caption{Big vertices arising from face $\alpha$ in 1-position (circle: vertex of $A$; square: vertex of $B$, i.e. big vertex).}
\label{fig:1pos}
\end{figure}
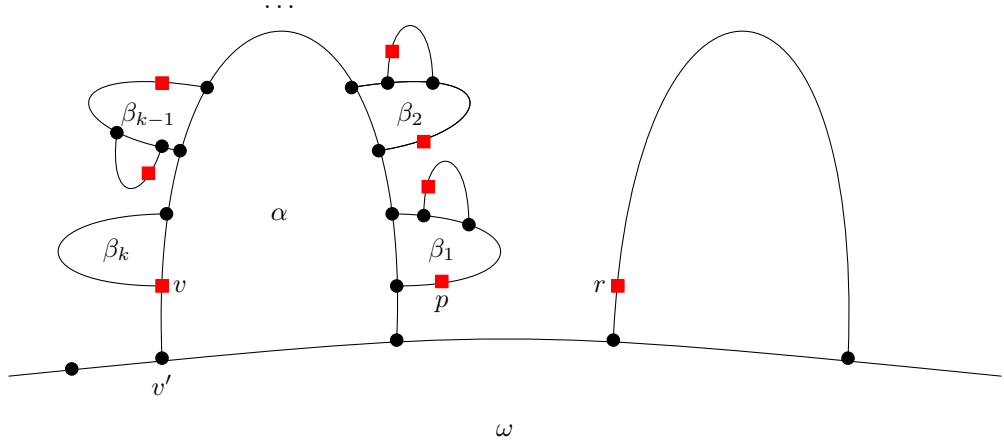

Consider the face $\omega$ and its descendants with respect to $\mathcal{T}_G$, and note that the color of their big vertices could already have been determined (with the exception of big vertices arising from $\alpha$). Also, as the face $\alpha$ is in 1-position, there can be a big vertex arising from $\alpha$, which is at distance 2 from $C(\omega)$. If this vertex exists, we denote it by $p$. By definition, $p$ is a big vertex arising from $\omega$, hence it is already colored. Also, since $v$ is a big vertex arising from $\omega$, it is also already colored.

By property (iii), the neighbors of $v'$, different from $v$, are not in $B$, and there can be at most one big vertex different from $v$ at distance $4$ from $p$, which is already colored. If such a vertex exists, we denote it by $r$. All other big vertices that are already colored are at distance at least $8$ from the vertices arising from $\alpha$.

Similarly as in Case 1 (considering the distance from $v$ and using property (iii)), we find the position of big vertices that possibly arise from $\alpha$, and depict them in Figure~\ref{fig:case2}. Note that some of the big vertices indicated in this figure may not exist in $G$.

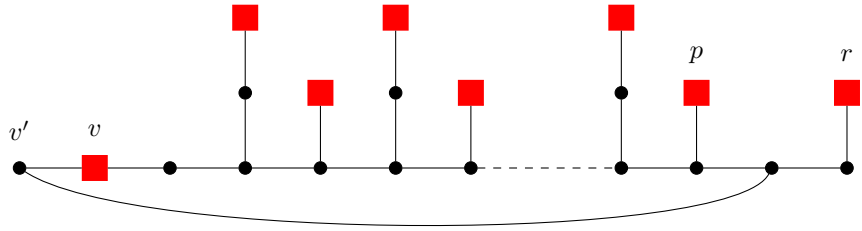
\begin{figure}[ht]
\begin{center}
\begin{tikzpicture}[scale=1]
\node at (-1,0) [circle,draw=black,fill=black,scale=0.5](u0){};
\node at (-1,0.5) {$v'$};
\node at (0,0.5) {$v$};
\node at (0,0) [regular polygon, regular polygon sides=4,draw=red,fill=red,scale=1](u){};
\node at (1,0) [circle,draw=black,fill=black,scale=0.5](u1){};
\node at (2,0) [circle,draw=black,fill=black,scale=0.5](u2){};
\node at (2,1) [circle,draw=black,fill=black,scale=0.5](u3){};
\node at (2,2) [regular polygon, regular polygon sides=4,draw=red,fill=red,scale=1](u4){};
\node at (3,0) [circle,draw=black,fill=black,scale=0.5](u5){};
\node at (3,1) [regular polygon, regular polygon sides=4,draw=red,fill=red,scale=1](u6){};
\node at (4,0) [circle,draw=black,fill=black,scale=0.5](u7){};
\node at (4,1) [circle,draw=black,fill=black,scale=0.5](u8){};
\node at (4,2) [regular polygon, regular polygon sides=4,draw=red,fill=red,scale=1](u9){};
\node at (5,0) [circle,draw=black,fill=black,scale=0.5](u10){};
\node at (7,0) [circle,draw=black,fill=black,scale=0.5](u12){};
\node at (5,1) [regular polygon, regular polygon sides=4,draw=red,fill=red,scale=1](u11){};
\node at (7,1) [circle,draw=black,fill=black,scale=0.5](u13){};
\node at (7,2) [regular polygon, regular polygon sides=4,draw=red,fill=red,scale=1](u14){};
\node at (8,0) [circle,draw=black,fill=black,scale=0.5](u15){};
\node at (8,1) [regular polygon, regular polygon sides=4,draw=red,fill=red,scale=1](u16){};
\node at (8,1.5) {$p$};
\node at (9,0) [circle,draw=black,fill=black,scale=0.5](u17){};
\node at (10,0) [circle,draw=black,fill=black,scale=0.5](u18){};
\node at (10,1) [regular polygon, regular polygon sides=4,draw=red,fill=red,scale=1](u19){};
\node at (10,1.5) {$r$};

\draw (u0) -- (u) -- (u1) -- (u2) -- (u3) -- (u4);
\draw (u2) -- (u5) -- (u6);
\draw (u5) -- (u7) -- (u8) -- (u9);
\draw (u7) -- (u10) -- (u11); 
\draw [style=dashed] (u10) -- (u12);
\draw (u12) -- (u13) -- (u14);
\draw (u12) -- (u15) -- (u16);
\draw (u15) -- (u17) .. controls (8.5,-1) and (.5,-1) .. (u0);
\draw (u17) -- (u18) -- (u19);

\end{tikzpicture}
\end{center}
\caption{\label{fig:case2} Case 2 ($\alpha$ is in 1-position).}
\end{figure}

Since, by property (iii), the vertices $v$, $p$ and $r$ (see Figure~\ref{fig:case2}) are pairwise at distance 4 from each other, exactly one of them has a very big color.  
Hence, in a similar way as in Case 1, we only consider three subcases, depending on the colors of the vertices $v$, $p$ and $r$. 
%We can assume, without loss of generality, that if the color of $v$, $p$ or $r$ is very big, then it is $6$; also, we may assume that $f(v)=4$, if $v$ is not very big. 
These three subcases give the way to deal with every possible color configuration for $v$, $p$ and $r$.
(Also, when $r$ or $p$ does not exist and no vertex among $v$, $p$ and $r$ is very big, it is possible to deal with this situation by considering that a non existing vertex among $v$, $p$ and $r$ is very big.) The various cases are described in the following table:

\begin{center}
\begin{tabular}{c|c|c|c}
Subcase & a & b & c\\\hline
$f(v)$ & 6 & 4 & 4\\
$f(p)$ & 4 & 6 & 5\\
$f(r)$ & 5 & 5 & 6\\
\end{tabular}
\end{center}

\medskip

Note that if the length $n$ of $C(\alpha)$ is $4$, there are no big vertices arising from $\alpha$, and this case is trivially resolved. Hence, we suppose that $n\ge6$.

\medskip

\textbf{Subcase 2.a.} $f(v)=6$, $f(p)=4$ and $f(r)=5$. 
We present patterns that define the coloring of the big vertices with respect to different lengths $n$ of the cycle $C(\alpha)$. The patterns give $f$-values of the big vertices in order of their presentation in Figure~\ref{fig:case2}. As in Case 1, if some of the big vertices that arise from $\alpha$ do not exist, we simply skip the corresponding values in the pattern. Note that the numbers between vertical bars are to be repeated $k$ times (case $k=0$ included). Note also that the last value in the patterns represents $f(p)=4$. 
\medskip

\bigskip

Pattern for $n=4k+8$:
\begin{tabular}{cc|cccc|cc}
%\hline
4 && 7 & & 6 & & 7\\
& 5 & & 4 && 5 && 4\\%\hline
\end{tabular}

\bigskip

Pattern for $n=4k+10$:
\begin{tabular}{cc|cccc|cccc}
%\hline
4 && 7 && 6 && 4 && 5\\
& 5 && 4 && 5 && 7 && 4\\%\hline
\end{tabular}

\bigskip

Note that because $f(v)=6$, the first four and the last three values (including the value of $p$) in the patterns cannot be $6$. In each of the patterns, one can check that property (iv) holds. 

Cases $k=0$, where the values between vertical bars are omitted, cover $n\in \{8,10\}$, 
while the pattern for length $n=6$ is:
\begin{tabular}{cc}
%\hline
5 & \\
& 4 \\%\hline
\end{tabular}

\medskip

\textbf{Subcase 2.b.} $f(v)=4$, $f(p)=6$ and $f(r)=5$.  We present patterns that define the coloring of the big vertices with respect to different lengths $n$ of the cycle $C(\alpha)$.  Note that the numbers between vertical bars are to be repeated $k$ times (case $k=0$ included). The last value in the patterns represents $f(p)=6$. 

\medskip

Pattern for $n=4k+6$:
\begin{tabular}{|cccc|cc}
%\hline
6 & & 7 & & 5\\
& 5 && 4 && 6\\%\hline
\end{tabular}

\bigskip
Pattern for $n=4k+8$:
\begin{tabular}{cc|cccc|cc}
%\hline
7 && 6 && 7 && 4\\
& 5 && 4 && 5 && 6\\%\hline
\end{tabular}

\bigskip

Note that $f(v)=4$ implies that patterns must avoid having the first two values equal to $4$. Cases $n\in\{6,8\}$ are covered by the above patterns with $k=0$ (i.e., removing the values between vertical bars). 

\medskip

\textbf{Subcase 2.c.} $f(v)=4$, $f(p)=5$ and $f(r)=6$. The following two patterns define the coloring of the big vertices (the numbers between vertical bars are to be repeated $k$ times, with $k=0$ included, covering $n=6$ and $n=8$). The last value in the patterns represents $f(p)=5$. 

\bigskip

Pattern for $n=4k+6$:
\begin{tabular}{|cccc|cc}
%\hline
7 & & 6 & & 7&\\
& 5 && 4 && 5\\%\hline
\end{tabular}
\bigskip

Pattern for $n=4k+8$:
\begin{tabular}{|cccc|cccc}
%\hline
7 && 6 && 5 && 4 &\\
&5 && 4 && 7 && 5\\%\hline
\end{tabular}

\bigskip

In the above patterns, in particular, the first two values cannot be $4$, as $f(v)=4$, and  the last three values cannot be $6$, as $f(r)=6$.

\bigskip

\textbf{Case 3.} $\alpha$ is in 2-position, see Figure~\ref{fig:2pos}.

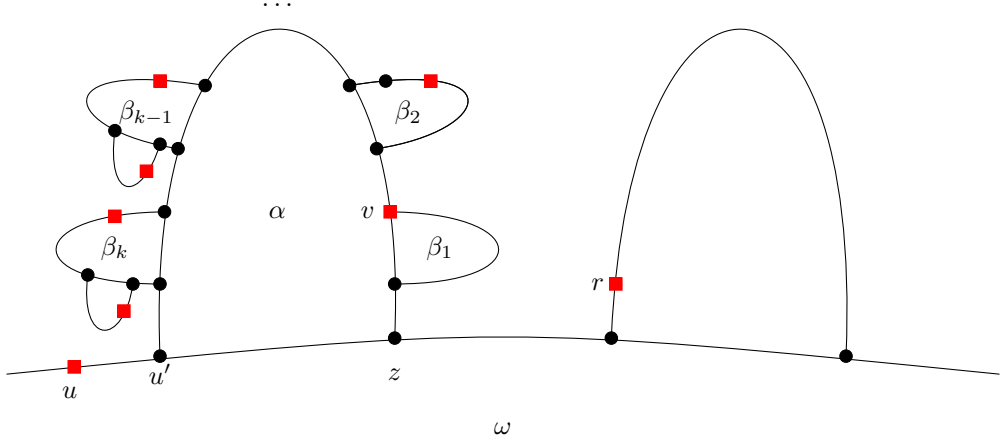
\begin{figure}[t]

\begin{center}

\begin{tikzpicture}[scale=1.2]
\draw  (-3,-0.2) .. controls (2.5,0.35) .. (8,-0.2);
\draw  (-1.3,0) .. controls (-1.5,4.8) and (1.5,4.8) .. (1.3,0.2);
\draw  (3.7,0.2) .. controls (3.9,4.8) and (6.5,4.8) .. (6.3,0);
\draw  (-1.3,0.8) .. controls (-2.9,0.8) and (-2.8,1.6) .. (-1.2,1.6);
\draw  (-1.1,2.3) .. controls (-2.4,2.5) and (-2.6,3.3) .. (-0.8,3);
\draw  (1.3,0.8) .. controls (2.9,0.8) and (2.8,1.6) .. (1.2,1.6);
\draw  (1.1,2.3) .. controls (2.4,2.5) and (2.6,3.3) .. (0.8,3);
\draw  (1.1,2.3) .. controls (2.4,2.5) and (2.6,3.3) .. (0.8,3);

\draw  (-1.8,2.5) .. controls (-1.9,1.7) and (-1.5,1.7) .. (-1.3,2.35);
\node at (-1.8,2.5) [circle,draw=black,fill=black,scale=0.5]{};
\node at (-1.3,2.35) [circle,draw=black,fill=black,scale=0.5]{};
\node at (-1.45,2.05) [regular polygon, regular polygon sides=4,draw=red,fill=red,scale=0.5]{};

\node at (1.2,3.05) [circle,draw=black,fill=black,scale=0.5]{};
\node at (1.7,3.05) [regular polygon, regular polygon sides=4,draw=red,fill=red,scale=0.5]{};

\draw  (-2.1,0.9) .. controls (-2.2,0.1) and (-1.7,0.1) .. (-1.6,0.8);
\node at (-2.1,0.9)  [circle,draw=black,fill=black,scale=0.5]{};
\node at (-1.6,0.8) [circle,draw=black,fill=black,scale=0.5]{};
\node at (-1.7,0.5) [regular polygon, regular polygon sides=4,draw=red,fill=red,scale=0.5]{};

\node at (-1.3,0) [circle,draw=black,fill=black,scale=0.5]{};
\node at (1.3,0.2) [circle,draw=black,fill=black,scale=0.5]{};
\node at  (3.7,0.2) [circle,draw=black,fill=black,scale=0.5]{};
\node at  (6.3,0) [circle,draw=black,fill=black,scale=0.5]{};
\node at (-1.25,1.6) [circle,draw=black,fill=black,scale=0.5]{};
\node at (-1.1,2.3) [circle,draw=black,fill=black,scale=0.5]{};
\node at (-0.8,3) [circle,draw=black,fill=black,scale=0.5]{};

\node at (1.3,0.8) [circle,draw=black,fill=black,scale=0.5]{};
\node at (1.25,1.6) [regular polygon, regular polygon sides=4,draw=red,fill=red,scale=0.5]{};
\node at (1.1,2.3) [circle,draw=black,fill=black,scale=0.5]{};
\node at (0.8,3) [circle,draw=black,fill=black,scale=0.5]{};

\node at (-1.3,0.8) [circle,draw=black,fill=black,scale=0.5]{};
\node at (-1.3,3.05) [regular polygon, regular polygon sides=4,draw=red,fill=red,scale=0.5]{};

\node at (-1.8,1.55) [regular polygon, regular polygon sides=4,draw=red,fill=red,scale=0.5]{};

\node at (-2.25,-0.12) [regular polygon, regular polygon sides=4,draw=red,fill=red,scale=0.5]{};
\node at (3.75,0.8) [regular polygon, regular polygon sides=4,draw=red,fill=red,scale=0.5]{};
\node at (3.55,0.8) {$r$};
\node at (-2.3,-0.4) {$u$};
\node at (-1.3,-0.2) {$u'$};
\node at (1.3,-0.2) {$z$};
\node at (1,1.6) {$v$};
\node at (2.5,-0.8) {$\omega$};
\node at (0,1.6) {$\alpha$};
\node at (1.8,1.2) {$\beta_{1}$};
\node at (1.45,2.7) {$\beta_{2}$};
\node at (-1.8,1.2) {$\beta_{k}$};
\node at (-1.45,2.7){$\beta_{k-1}$};
\node at (0,3.9) {$\ldots$};

\end{tikzpicture}
\end{center}
\caption{Big vertices arising from face $\alpha$ in 2-position (circle: vertex of $A$; square: vertex of $B$, i.e. big vertex).}
\label{fig:2pos}
\end{figure}

Again we consider the face $\omega$ and its descendants with respect to $\mathcal{T}_G$, and note that the color of their big vertices could already have been determined (with the exception of big vertices arising from $\alpha$). Also, by the definition of 2-position of the face $\alpha$, $u$ and $v$ are already colored. There can be at most one additional vertex, which is already colored and is at distance $6$ to a big vertex arising from $\alpha$ that is not yet colored; if this vertex exists, we denote it by $r$, see Figure~\ref{fig:2pos}.

Similarly as in Case 1 (considering the distance from $v$ and using property (iii)), we find the possible positions of big vertices that arise from $\alpha$, and depict them in Figure~\ref{fig:case3}. Note that some of the big vertices indicated in this figure may not exist in $G$. 

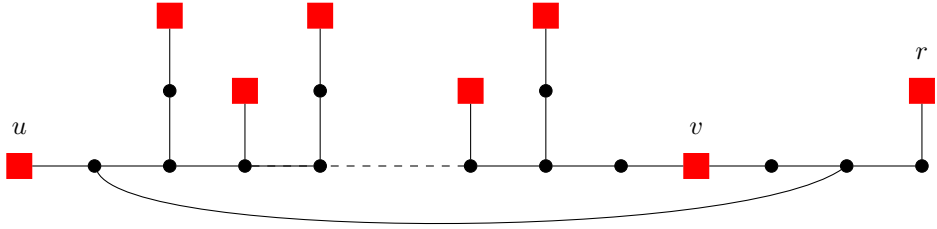
\begin{figure}[ht]
\begin{center}
\begin{tikzpicture}[scale=1]
\node at (0,0.5) {$u$};
\node at (0,0) [regular polygon, regular polygon sides=4,draw=red,fill=red,scale=1](u){};
\node at (1,0) [circle,draw=black,fill=black,scale=0.5](u1){};
\node at (2,0) [circle,draw=black,fill=black,scale=0.5](u2){};
\node at (2,1) [circle,draw=black,fill=black,scale=0.5](u3){};
\node at (2,2) [regular polygon, regular polygon sides=4,draw=red,fill=red,scale=1](u4){};
\node at (3,0) [circle,draw=black,fill=black,scale=0.5](u5){};
\node at (3,1) [regular polygon, regular polygon sides=4,draw=red,fill=red,scale=1](u6){};
\node at (4,0) [circle,draw=black,fill=black,scale=0.5](u7){};
\node at (4,1) [circle,draw=black,fill=black,scale=0.5](u8){};
\node at (4,2) [regular polygon, regular polygon sides=4,draw=red,fill=red,scale=1](u9){};
\node at (6,0) [circle,draw=black,fill=black,scale=0.5](u10){};
\node at (6,1) [regular polygon, regular polygon sides=4,draw=red,fill=red,scale=1](u11){};
\node at (7,0) [circle,draw=black,fill=black,scale=0.5](u12){};
\node at (7,1) [circle,draw=black,fill=black,scale=0.5](u13){};
\node at (7,2) [regular polygon, regular polygon sides=4,draw=red,fill=red,scale=1](u14){};
\node at (8,0)  [circle,draw=black,fill=black,scale=0.5] (u14b){};
\node at (9,0) [regular polygon, regular polygon sides=4,draw=red,fill=red,scale=1](u15){};
\node at (9,.5) {$v$};
\node at (10,0) [circle,draw=black,fill=black,scale=0.5](u16){};
\node at (11,0) [circle,draw=black,fill=black,scale=0.5](u17){};
\node at (12,0) [circle,draw=black,fill=black,scale=0.5](u18){};
\node at (12,1) [regular polygon, regular polygon sides=4,draw=red,fill=red,scale=1](u19){};
\node at (12,1.5) {$r$};
%\node at (12,0) [circle,draw=black,fill=black,scale=0.5](u20){};

\draw (u) -- (u1) -- (u2) -- (u3) -- (u4);
\draw (u2) -- (u5) -- (u6);
\draw (u5) -- (u7) -- (u8) -- (u9);
\draw [style=dashed] (u5) -- (u10);
\draw (u11) -- (u10) -- (u12) -- (u13) -- (u14);
\draw (u12) -- (u15) -- (u16) -- (u17) .. controls (9.5,-1) and (1.5,-1) .. (u1);
\draw (u17) -- (u18) -- (u19);
%\draw (u18) -- (u20);
\end{tikzpicture}
\end{center}
\caption{\label{fig:case3} Case 3 ($\alpha$ is in 2-position).}
\end{figure}

Since the vertices $u$, $v$ and $r$ (see Figures~\ref{fig:2pos} and~\ref{fig:case3}) are pairwise at distance 4 from each other, exactly one of them has a very big color.  
Hence, in a similar way as in Case 1, we only we consider three subcases, depending on which vertex among $u$, $v$ and $r$ is very big. 
%We can assume, without loss of generality, that if the color of $u,v$ or $r$ is very big, then it is $6$; also, we may assume that $f(v)=4$, if $v$ is not very big. 
These three subcases give the way to deal with every possible color configuration for $u$, $v$ and $r$.
(Also, in the case $r$ does not exist and no vertex among $u$ and $v$ is very big, it is possible to deal with this situation by considering that $r$ is very big.)
The cases are described in the following table:

\begin{center}
\begin{tabular}{c|c|c|c}
Subcase & a & b & c\\\hline
$f(u)$ & 4 & 6 & 5\\
$f(v)$ & 6 & 4 & 4\\
$f(r)$ & 5 & 5 & 6\\
\end{tabular}
\end{center}

\medskip

If $n=4$ there are no big vertices arising from $\alpha$. Hence, we suppose that $n\ge6$.

\medskip

\textbf{Subcase 3.a.} $f(u)=4$, $f(v)=6$ and $f(r)=5$. The following two patterns define the coloring of the big vertices (the numbers between vertical bars are to be repeated $k$ times, including $k=0$).

\medskip

Pattern for $n=4k+8$:
\begin{tabular}{|cccc|ccc}
%\hline
7 & & 6 & & 7 & & 4\\
& 5 & & 4 && 5 &\\%\hline
\end{tabular}
\bigskip

Pattern for $n=4k+10$:
\begin{tabular}{|cccc|ccccc}
%\hline
7 && 6 && 5 && 5 && 5\\
& 5 && 4 && 7 && 4 &\\%\hline
\end{tabular}

\medskip 

%Note that $f(v)=6$ and $f(u)=4$ enforce the last four numbers of the patterns being different from $6$, and the first two numbers different from $4$ and $6$. 
Note that because $f(v)=6$ and $f(u)=4$, the last four values in the patterns cannot be $6$ and the first two values in the patterns cannot be $4$ or $6$. In each of the patterns, one can check that property (iv) holds. 

\medskip

For $n=6$, there can be at most one big vertex arising from $\alpha$, and we can color it by $5$. For $n=4$ there are no big vertices arising from $\alpha$. 

\medskip

\textbf{Subcase 3.b.} $f(u)=6$, $f(v)=4$ and $f(r)=5$. The following two patterns define the coloring of the big vertices (the numbers between vertical bars are to be repeated $k$ times, including $k=0$).

\medskip

Pattern for $n=4k+8$:
\begin{tabular}{cc|cccc|c}
%\hline
4 & & 7 & & 6 & & 7\\
& 5 && 4 && 5 &\\%\hline
\end{tabular}

\bigskip
Pattern for $n=4k+10$:
\begin{tabular}{cccc|cccc|c}
%\hline
5 && 7 && 6 && 7 && 6\\
& 4 && 5 && 4 && 5 &\\%\hline
\end{tabular}

\medskip 

%Note that $f(u)=6$ and $f(v)=4$ enforce the first four numbers of the patterns being different from $6$, and the last two numbers different from $4$. 
Note that because $f(u)=6$ and $f(v)=4$, the first four values in the patterns cannot be $6$ and the last two values in the patterns cannot be $4$. In each of the patterns, one can check that property (iv) holds. 

For $n=6$, there can be at most one big vertex arising from $\alpha$, and we can color it by $5$. For $n=4$ there are no big vertices arising from $\alpha$. 

\medskip

\textbf{Subcase 3.c.} $f(u)=5$, $f(v)=4$ and $f(r)=6$. Finally, the following two patterns define the coloring of the big vertices (the numbers between vertical bars are to be repeated $k$ times, $k=0$ included).

Pattern for $n=4k+6$:
\begin{tabular}{|cccc|c}
%\hline
7 & & 6 & & 7\\
& 4 && 5 &\\%\hline
\end{tabular}
\bigskip

Pattern for $n=4k+8$:
\begin{tabular}{|cccc|ccc}
%\hline
7 && 6 && 4 && 5\\
& 4 && 5 && 7 &\\%\hline
\end{tabular}

\medskip 

%Note that $f(r)=6$ implies that the patterns must not use color $6$ in the first two values. Also, $f(u)=5$ implies that the first two values in the patterns are not $5$: the last two values are not $4$ because of $f(v)=4$.  
Note that because $f(r)=6$, $f(v)=4$ and $f(u)=5$, the first two values in the patterns cannot be $5$ or $6$ and the last two values in the patterns cannot be $4$. In each of the patterns, one can check that property (iv) holds. 

It is straightforward to see that in all of the above patterns property (iv) is satisfied, and that the resulting coloring $f$ is a packing coloring that uses $7$ colors. This completes the proof. 
 \qed

%%%%%%%%%%%
%%%%%%%%%%% END PROOF THEREM
%%%%%%%%%% 

\bigskip 

We complete this section by showing that the upper bound of $7$ on the packing chromatic number of 2-connected bipartite subcubic outerplanar graphs is sharp. 

\begin{prop}
There exists a $2$-connected bipartite subcubic outerplanar graph $G$ such that $\chi_{\rho}(G)\ge 7$.
\end{prop}
\begin{proof}
Let $\mathcal{T}$ be the infinite binary tree. Sloper~\cite{Slo} has proven that $\chi_{\rho}(\mathcal{T})=7$. A consequence is that there exists a finite subcubic tree $T$ such that $\chi_{\rho}(T)=7$. Let $d$ be the depth of $T$. Finally, let $k=2d+2$.

Let $x$ and $y$ be two adjacent vertices of degree $2$ in a graph $G$. {\em Adding a $k$-cycle on $x$ and $y$} is an operation that consists in adding a path of $k-2$ vertices to $G$ and joining one endvertex of the path to $x$ and the other endvertex to $y$.
Let $G_1$ be a cycle of order $k$. Let $G_{i+1}$, $i\ge 1 $, be the graph obtained from $G_i$ by adding a $k$-cycle on every adjacent pair of vertices of degree $2$ in $G_i$, where we arrange these pairs in such a way that each vertex of degree $2$ belongs to one adjacent pair; this can be done by adding $k$-cycles on adjacent vertices of degree $2$ following the outer cycle of $G_i$. In this way, in $G_{i+1}$ there does not remain any vertex of degree $2$ from $V(G_i)$. Note that, by construction, $G_i$ is a $2$-connected bipartite subcubic outerplanar graph for every integer $i\ge 1$.

Let $u$ be a vertex that belongs to $G_1$ in the construction of $G_{k}$. Note that the set $\{v\in V(G_{k})|\ d(u,v)\le d \}$ induces a subcubic tree containing $T$ as an induced subgraph. Thus, since every graph has a packing chromatic number larger than or equal to the packing chromatic number of any of its (induced) subgraphs, we derive $\chi_{\rho}(G_{k})\ge 7$.
\end{proof}

%%%%%%%%
%%%%%%%%%
%%%%%%%%%
%%%  1, 3, ...., 3    P A C K I N G    C O L O R I N G
%%%%%%%%%%
%%%%%%%%%%
%%%%%%%%%%

\section{$(1,3,\ldots,3)$-packing coloring of bipartite outerplanar
graphs} 
\label{sec:133}
%%%%%%%%%%%%%

In this section, we need to extend the definition of $\mathcal{T}_G$ in order to have an underlying tree even if the outerplanar graph is not 2-connected.
For an outerplanar graph $G$, let $D=\{ u\in V(G)|\  u\notin C(v),\ v\in V(\mathcal{T}_G) \}$ and let $A=V(G)\setminus D$.
Note that the graph with vertex set $V(\mathcal{T}_G)\cup D$ and edge set $E( \mathcal{T}_G)\cup E(G[D])$ is a forest.
We construct $\mathcal{L}_G$ from the forest with vertex set $V(\mathcal{T}_G)\cup D$ and edge set $E( \mathcal{T}_G)\cup E(G[D])$ as follows. 
First, for each bridge $uv$ of $G$ such that $u\in A$ and $v\in D$, we add an edge to $\mathcal{L}_G$ between $v$ and an arbitrary face $\alpha$ containing $u$. 
Second, for each bridge $uv$ of $G$ such that $u\in A$ and $v\in A$, we add an edge to $\mathcal{L}_G$ between an arbitrary face $\alpha$ containing $u$ and an arbitrary face $\beta$ containing $v$. Third, let $G'$ be the graph obtained from $G$ by removing the bridges. % from $E(G)$. 
For a cut vertex $u$ of $G'$, let $B_1$, $\ldots$, $B_k$ be the maximal 2-connected components of $G'$ containing $u$ and let $\alpha_i(u)$ be a face chosen arbitrarily among the faces from $B_i$ containing $u$, $1\le i\le k$. For each cut vertex $u$ and each integer $i$ between 2 and $k$, we add an edge to $\mathcal{L}_G$ between $\alpha_1(u)$ and $\alpha_i(u)$.

It is easily seen that the graph $\mathcal{L}_G$ is a tree for any outerplanar graph $G$. %This can be remarked by using the definition of cut vertex and bridge.
Figure~\ref{graphL} illustrates an example construction of the tree $\mathcal{L}_G$ for an outerplanar graph $G$.

\begin{figure}[ht]
\begin{center}
\begin{tikzpicture}[scale=1.8]
\draw (1,0) -- (1.5,-0.5);
\draw (1,0) -- (1.5,0.5);
\draw (2,0) -- (1.5,-0.5);
\draw (2,0) -- (1.5,0.5);
\draw (1,0) -- (7,0);
\draw (6,0) -- (7,-0.5);
\draw (3,0) -- (3.5,-0.5);
\draw (3,0) -- (3.5,0.5);
\draw (4,0) -- (3.5,-0.5);
\draw (4,0) -- (3.5,0.5);
\draw (4,0) -- (4.5,-0.5);
\draw (4,0) -- (4.5,0.5);
\draw (5,0) -- (4.5,-0.5);
\draw (5,0) -- (4.5,0.5);

\draw[color=red, dashed,ultra thick] (1.5,0.25) -- (1.5,-0.25);
\draw[color=red, dashed,ultra thick] (1.5,0.25) -- (3.5,-0.25);
\draw[color=red, dashed,ultra thick] (3.5,0.25) -- (3.5,-0.25);
\draw[color=red, dashed,ultra thick] (3.5,0.25) -- (4.5,0.25);
\draw[color=red, dashed,ultra thick] (4.5,0.25) -- (4.5,-0.25);
\draw[color=red, dashed,ultra thick] (4.5,0.25) -- (6,0);
\draw[color=red, dashed,ultra thick] (6,0) -- (7,0);
\draw[color=red, dashed,ultra thick] (6,0) -- (7,-0.5);

\node at (1,0) [circle,draw=black,fill=black,scale=0.5]{};
\node at (1.5,-0.5) [circle,draw=black,fill=black,scale=0.5]{};
\node at (1.5,0.5) [circle,draw=black,fill=black,scale=0.5]{};
\node at (2,0) [circle,draw=black,fill=black,scale=0.5]{};
\node at (3,0) [circle,draw=black,fill=black,scale=0.5]{};
\node at (3.5,-0.5) [circle,draw=black,fill=black,scale=0.5]{};
\node at (3.5,0.5) [circle,draw=black,fill=black,scale=0.5]{};
\node at (4,0) [circle,draw=black,fill=black,scale=0.5]{};
\node at (4.5,-0.5) [circle,draw=black,fill=black,scale=0.5]{};
\node at (4.5,0.5) [circle,draw=black,fill=black,scale=0.5]{};
\node at (5,0) [circle,draw=black,fill=black,scale=0.5]{};
\node at (6,0) [circle,draw=black,fill=black,scale=0.5]{};
\node at (7,0) [circle,draw=black,fill=black,scale=0.5]{};
\node at (7,-0.5) [circle,draw=black,fill=black,scale=0.5]{};

\node at (1.5,-0.25) [regular polygon, regular polygon sides=4,draw=red,fill=red,scale=0.5]{};
\node at (1.5,0.25) [regular polygon, regular polygon sides=4,draw=red,fill=red,scale=0.5]{};
\node at (3.5,-0.25) [regular polygon, regular polygon sides=4,draw=red,fill=red,scale=0.5]{};
\node at (3.5,0.25) [regular polygon, regular polygon sides=4,draw=red,fill=red,scale=0.5]{};
\node at (4.5,-0.25) [regular polygon, regular polygon sides=4,draw=red,fill=red,scale=0.5]{};
\node at (4.5,0.25) [regular polygon, regular polygon sides=4,draw=red,fill=red,scale=0.5]{};

\end{tikzpicture}
\end{center}
\caption{\label{graphL} The graph $\mathcal{L}_G$ for an outerplanar graph $G$ (circle: vertex of $G$; square: vertex of $\mathcal{T}_G$; line: edge of $G$; dashed line: edge of $\mathcal{L}_G$).}
\end{figure}
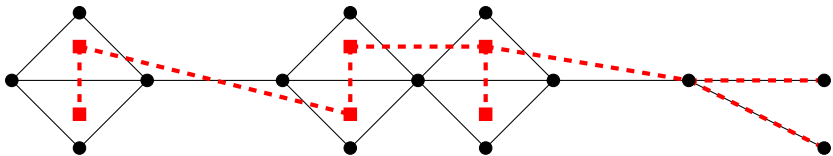

First, we prove Theorem~\ref{th:133}, which states that a bipartite outerplanar graph with maximum degree $\Delta$ bounded by $k$ is $(1,3,\ldots,3)$-colorable, where $3$ appears $k$ times in the sequence, $k\ge 3$.

\medskip 

\noindent {\bf Proof of Theorem~\ref{th:133}.}

In the construction of the $S$-packing coloring, the vertices with color $1$
will form an independent set and the vertices with color $a_i$,
$i\in\{1,\ldots,k\}$ will form a $3$-packing.
%When we do not precise that we color a vertex with color $1$, we color it with a color $a_i$, $i\in\{1,\ldots,k\}$.
The proof is by induction on the order of $\mathcal{L}_G$. If $G$
has more than one vertex, vertices colored by $1$ will correspond to one
part of the bipartition of $G$.

For the induction, suppose that $\mathcal{L}_G$  contains only one
vertex $u$. If $u\in D$, then it suffices to color $u$ with color $1$.
If $u\in V(\mathcal{T}_G)$, then we have to color the cycle $C(u)$.
We can clearly color it with colors  $1$, $a_1$, $a_2$ and $a_3$
by (repeatedly) using the pattern $1,a_1,1,a_2$, and using the color $1,a_3$ for the
last two vertices when the length of the cycle is not divisible by
four.

Now, consider a graph $G$ such that $\mathcal{L}_G$  has order
$n+1$, $n\ge 1$. Let $u$ be a leaf of $\mathcal{L}_G$ and let $G'$ be the
following graph:
$$
G' = \left\{
    \begin{array}{ll}
        G-u & \mbox{if } u \in D \\
        G-B(u) & \mbox{if } u\in V(\mathcal{T}_G)\mbox{;}
    \end{array}
\right.
$$
where $B(u)$ is the set of the vertices of $V(G)\setminus D$, which
belong to $C(u)$ but no other inner face of $G$. By induction, we
can color the vertices of $G'$, since $\mathcal{L}_{G'}$ has order
$n$, and it suffices to extend the coloring of $G'$ to the uncolored
vertices of $G$. Let $v$ be the neighbor of $u$ in
$\mathcal{L}_{G'}$.
\begin{description}
\item[Case 1] $u\in D$.

Let $v'$ denote the neighbor of $u$ in $C(v)$ if $v\in
 \mathcal{T}_G$, and $v'=v$ if $v\in D$. If $v'$ is not colored by
$1$, then we can color $u$ by $1$ and we are done.
Otherwise, when $v'$ is colored by color $1$, then note that $v'$
has at most $k-1$ colored neighbors, whose neighbors are all colored
by $1$. Therefore, we can color $u$ by a color $a_i$, which is not
given to any vertices of $N_G(v')$, since other vertices with color $a_i$
are at distance at least 4 from $u$.

\item[Case 2] $u\in V(\mathcal{T}_G)$ and $v\in D$.

Let $u'$ be the neighbor of $v$ in $C(u)$. Firstly, if $v$ is
colored by $1$, then, since $\Delta(G)\le k$, at most $k-1$
neighbors of $v$ are colored in $G$, and we can color $u'$ by a
color $a_i$ not used in other neighbors of $v$. We can color the
remaining uncolored vertices of $C(u)$ by using the color $1$ and
three more colors (proceeding in the same way as in the coloring of
a cycle described in the initial step of the induction). Secondly,
if $v$ is not colored by color $1$, then we color $u'$ with color $1$.
We can again extend the coloring to the remaining vertices of $C(u)$
using color $1$ and three more colors in an analogous way as in the
initial step of the induction. We just need to avoid that a neighbor of $u'$ is assigned the same color as $v$, which is always possible since $k\ge 3$.

\item[Case 3] $u\in V(\mathcal{T}_G)$ and $v\in V(\mathcal{T}_G)$.

\item[Subcase 3.a] $|V(C(u))\cap V(C(v))|=0$.

In this case, a vertex $u'\in C(u)$ is adjacent to a vertex $v'\in
C(v)$. For the proof of this case, we follow the same steps as in
Case 2, where the vertex $v'$, defined in the previous sentence,
plays the role of $v$ in Case 2.

\item[Subcase 3.b] $|V(C(u))\cap V(C(v))|=1$.

Let $\{w\}=V(C(u))\cap V(C(v))$. Suppose $w$ is colored by color
$1$. Since $w$ is in the uncolored facial cycle $C(u)$, at most $k-2$
neighbors of $w$ are colored so far  in $G$. Thus, we can give two
(distinct) colors not used by neighbors of $w$ to the two neighbors of $w$ in $C(u)$ and easily extend the coloring to the remaining uncolored vertices of
$C(u)$ using these two colors, color 1 and possibly a third color $a_i$ (that will be given to a vertex at distance at least $4$ from any vertex with color $a_i$ of another cycle). Now, if $w$ is not
colored with color $1$, then every neighbor of $w$ can be colored
with color $1$. It is again easy to color the remaining uncolored
vertices of $C(u)$ by using color $1$ and three more colors,
applying the pattern as in the initial step of the induction.

\item[Subcase 3.c] $|V(C(u))\cap V(C(v))|=2$.

Let $\{w_1,w_2\}=V(C(u))\cap V(C(v))$. Since $G$ is outerplanar,
$w_1$ and $w_2$ are adjacent in $G$. Without loss of generality
suppose $w_1$ is colored by color $1$. By the induction hypothesis $w_2$
is already colored. Let $x$ be the neighbor of $w_1$ (which is not
$w_2$) in $C(u)$. Since the vertex $w_1$ has at most $k-1$ colored
neighbors, we can color the vertex $x$ by a color $a_i$ that is not
used in any of the neighbors of $w_1$. The remaining vertices of
$C(u)$ can be colored by using color $1$ and three more colors (that of $w$, of $x$ and a third one) in
the same way as in the previous cases.
\end{description}
\qed

In the following proposition, we prove that Theorem~\ref{th:133} does not hold if in $S$ an integer $3$ is replaced by an integer $4$.

\begin{prop}
There exists a bipartite outerplanar graph $G$ with $\Delta(G)\le
k$, which is not $S$-packing colorable for the list
$S=(1,3,\ldots,3,4)$ containing $k-1$ times the integer $3$.
\end{prop}
\begin{proof}
%Let $T$ be the star graph of order $k+1$, i.e., a tree of diameter $2$ containing $k$ leaves and one central vertex $u$ of degree $k$. Adding a leaf on a vertex $v$ consist in adding a new vertex $w$ and the edge $vw$.
%Let $G$ be the graph obtained from $S$ by adding $k-1$ leaves on each vertex of degree $1$ of $T$.
Let $T$ be the complete $k$-ary tree of depth 5 and root $r$, and
suppose there exists an $S$-coloring $c$ of $T$, using
$S=(1,3,\ldots,3,4)$ as in the statement of the proposition.  Note
that there exists a vertex $x\in\{r\}\cup N(r)$ that is colored
by $1$. Since $x$ has $k$ neighbors, all must receive distinct
colors, different from $1$. In particular, there exists a neighbor $y$ of $x$ such that $c(y)=4$. Let $z$ be any neighbor of $x$ different from $y$. Clearly, $z$ must be colored by a color '3', while all neighbors of $z$ are colored by $1$.
Finally, the neighbors of the vertices in $N(z)$ must receive all colors different from $1$. In particular, there exists a vertex $u$ with $d(u,y)=4$, such that $c(u)=4$, which is a contradiction.

% Let $T$ be the complete $k$-ary tree of height 5 and root $r$, and
% suppose there exists an $S$-coloring $c$ of $T$, using
% $S=(1,3,\ldots,3,4)$ as in the statement of the proposition.  Note
% that there exists a vertex $x$, which is not a leaf, that is colored
% by $1$. Since $x$ has $k$ neighbors, all must receive distinct
% colors, different from $1$. In particular, there exists a neighbor
% $y$ of $x$ such that $c(y)=4$. If $y$ is a leaf of $T$, then
% consider the subtree $T'$ of $T$, in which $x$ is a leaf, and does
% not contain $y$. Clearly, the neighbor $z$ of $x$ in $T'$ must be
% colored by one of the colors $"3"$, while all neighbors of $z$ are
% colored by $1$. Finally, the neighbors of the vertices in $N(z)$
% must receive all colors different from $1$. In particular, there
% exists a vertex $u$ with $d(u,y)=4$, such that $c(u)=4$, which is a
% contradiction.
% 
% If $y$ is not a leaf, then it is either a root of $T$ or it is
% adjacent to a vertex $x$, which is a root of a subtree $T'$ of $T$,
% which is complete $k$-ary tree of height at least $4$. Since
% $c(x)=1$, we derive that in $T'$ there must be vertices $z$ with
% $d(x,z)=3$, such that $c(z)=4$. Since $d(y,z)=4$, this is again a
% contradiction with $c$ being an $S$-packing coloring. Therefore $T$
% is a bipartite outerplanar graph that is not
% $(1,3,\ldots,3,4)$-packing colorable.
\qed

%%%%%%%%
%%%%%%%%%
%%%%%%%%%
%%%  1, 2, 2, 2   -   P A C K I N G    C O L O R I N G
%%%%%%%%%%
%%%%%%%%%%
%%%%%%%%%%

%%%%%%%%%%%%%%%%%%%%%%%%%%%%%%%%%%%
\section{$(1,2,2,2)$-packing coloring of subcubic outerplanar graphs}
\label{sec:1222}
%%%%%%%%%%%%%%%%%%%%%%%%%%%%%%%%%%%%

In this section, we first prove that a subcubic outerplanar graph $G$ is $(1,2,2,2)$-packing colorable when $G$ is triangle-free outerplanar.

\medskip 

\noindent {\bf Proof of Theorem~\ref{th:1222}.}

Let $G$ be a  subcubic outerplanar graph and let $D$ and $\mathcal{L}_G$ be defined as in Section \ref{sec:133}.
In this proof, the vertices with color $1$ will form an independent set and the vertices with color $a_i$, for each $i\in\{1,2,3\}$ will form a $2$-packing.
%When we do not precise that we color a vertex with color $1$, we color it with a color $a_i$, $i\in\{1,\ldots,k\}$.

By induction on the order of $\mathcal{L}_G$, we prove that there is an $S$-packing coloring of $G$. Suppose $\mathcal{L}_G$  has only one vertex $u$. If $u\in D$, then it suffices to color $u$ with color $1$. If $u\in V(\mathcal{T}_G)$, then we have to color a cycle of order $n$. We color it with colors  $1$, $a_1$, $a_2$ and $a_3$ using the pattern $1,a_1,1,a_2$ and using the color $a_3$ for the last vertex if $n\equiv1 \pmod{4}$, colors $1,a_3$ for the two last vertices if $n\equiv2 \pmod{4}$ or colors $1,a_3,a_2$ for the three last vertices if $n\equiv3 \pmod{4}$.

Now, consider a graph $G$ such that $\mathcal{L}_G$  has order
$n+1$, $n\ge1$. Let $u$ be a leaf of $\mathcal{L}_G$ and let $G'$ be the
following graph:
$$
G' = \left\{
    \begin{array}{ll}
        G-u & \mbox{if } u \in D \\
        G-B(u) & \mbox{if } u\in V(\mathcal{T}_G)\mbox{;}
    \end{array}
\right.
$$
where $B(u)$ is the set of the vertices of $V(G)\setminus D$, which
belong to $C(u)$ but no other inner face of $G$. By induction, we
can color the vertices of $G'$, since $\mathcal{L}_{G'}$ has order
$n$, and it suffices to extend the coloring of $G'$ to the uncolored
vertices of $G$. Let $v$ be the neighbor of $u$ in
$\mathcal{L}_{G'}$.
\begin{description}
\item[Case 1] $u\in D$.

Let $v'$ denote the neighbor of $u$ in $C(v)$ if $v\in \mathcal{T}_G$ and $v'=v$ if $v\in D$.
If $v'$ is not colored by $1$, then we can color $u$ by $1$ and we are done. Otherwise ($v'$ is colored by $1$), since $G$ is subcubic, $v'$ has at most two colored neighbors. Therefore, we can color $u$ by a color not given to vertices from $N_G(v')$.

\item[Case 2] $u\in \mathcal{T}_G$ and $v\in D$.

Let $u'$ be the neighbor of $v$ in $C(u)$.
First, if $v$ is colored by color $1$, then, since $\Delta(G)\le 3$, at most two neighbors of $v$ are colored in $G$ and we can easily color $u'$. We can color the remaining uncolored vertices of $C(u)$ by the color $1$ and the three remaining colors  (by proceeding as in the coloring of a cycle described in the initial step of the induction).
Second, if $v$ is not colored by color $1$, then we color $u'$ with color $1$. Since $u'$ is in $C(u)$ which is uncolored, exactly one neighbor of $u'$, namely $v$, is colored in $G$. Thus, we can give two different colors from $\{a_1,a_2,a_3\}$ to the neighbors of $u'$ in $C(u)$. We now extend the coloring to the remaining vertices of $C(u)$ using color $1$ and the three remaining colors (by proceeding as in the coloring of a cycle described in the initial step of the induction).

\item[Case 3] $u\in \mathcal{T}_G$ and $v\in \mathcal{T}_G$.

Since $G$ is subcubic, $|V(C(u))\cap V(C(v))|=0$ or $|V(C(u))\cap V(C(v))|=2$ (if $|V(C(u))\cap V(C(v))|=1$, then the common vertex would have degree at least 4).

\item[Subcase 3.a] $|V(C(u))\cap V(C(v))|=0$.

In this case, a vertex $u'\in C(u)$ is adjacent to a vertex $v'\in
C(v)$. For the proof of this case, we follow the same steps as in
Case 2, where the vertex $v'$, defined in the previous sentence,
plays the role of $v$ in Case 2.

\item[Subcase 3.b] $|V(C(u))\cap V(C(v))|=2$.

Let $\{w_1,w_2\}=V(C(u))\cap V(C(v))$.
Since $G$ is outerplanar, $w_1$ and $w_2$ are adjacent in $G$.
By the induction hypothesis, $w_1$ and $w_2$ are already colored. Let $x_1$ be the neighbor of $w_1$ (which is not $w_2$) in $C(u)$ and let $x_2$ be the neighbor of $w_2$ (which is not $w_1$) in $C(u)$. 
If $w_1$ has no neighbor with color $1$, then we recolor it with color $1$. 
If after this, $w_2$ has no neighbor with color $1$, then we recolor it with color $1$. 

Suppose that $4\le|C(u)|\le 5$. First, if one of $w_i$ is colored by $1$, say $w_1$, we can extend the coloring to $C(u)$ as follows. In this case, $x_1$ receives the color $a_i$, which is not used in the neighborhood of $w_1$, and we can color $x_2$ by $1$. In the case $|C(u)|=5$, the common neighbor of $x_1$ and $x_2$ receives the color that was given to the neighbor of $w_1$, which is not $w_2$, in $C(v)$. Second, assume without loss of generality that $w_i$ received color $a_i$ for $i\in\{1,2\}$. Note that by the above recoloring condition, $w_1$ and $w_2$ have neighbors, which are given color $1$. If $|C(u)|=4$, we can color $x_1$ and $x_2$ by colors $1$ and $a_3$, respectively.  Otherwise, if $|C(u)|=5$, we give color $1$ to vertices $x_1$ and $x_2$, and the common neighbor of $x_1$ and $x_2$ gets color $a_3$. 

So, let $|C(u)|>5$.
We color the uncolored vertices of $C(u)$ starting by coloring the vertex $x_1$. 
We color $x_1$ by using a color not given to $w_1$ and the colored neighbors of $w_1$.
If $w_2$ is not colored by $1$ we color $x_2$ by $1$. Otherwise, we color $x_2$ by a color not given to $w_2$ and the colored neighbors of $w_2$. The remaining vertices of $C(u)$ can be colored by the pattern described in the initial step of the induction, alternating $1,a_i,1,a_j$, and possibly completing the coloring of the cycle with the third color $a_k$, depending on the length of $C(u)$.
\end{description}
\end{proof}

We next present two examples, which show that Theorem~\ref{th:1222} is best possible. 
First, we prove that the result does not hold if the graph contains triangles.  

\begin{prop}
There exists a subcubic outerplanar graph, which is not $(1,2,2,2)$-packing colorable.
\end{prop}
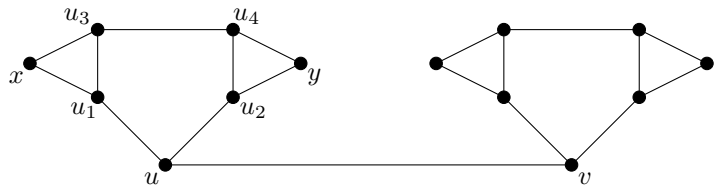
\begin{figure}[ht]
\begin{center}
\begin{tikzpicture}[scale=1.8]
\draw (0,0) -- (-0.5,0.5);
\draw (0,0) -- (0.5,0.5);
\draw (-0.5,0.5) -- (-0.5,1);
\draw (0.5,0.5) -- (0.5,1);
\draw (-0.5,1) -- (0.5,1);
\draw (-0.5,0.5) -- (-1,0.75);
\draw (-0.5,1) -- (-1,0.75);
\draw (0.5,0.5) -- (1,0.75);
\draw (0.5,1) -- (1,0.75);

\node at (0,0) [circle,draw=black,fill=black,scale=0.5]{};
\node at (-0.5,0.5) [circle,draw=black,fill=black,scale=0.5]{};
\node at (0.5,0.5) [circle,draw=black,fill=black,scale=0.5]{};
\node at (-0.5,1) [circle,draw=black,fill=black,scale=0.5]{};
\node at (0.5,1) [circle,draw=black,fill=black,scale=0.5]{};
\node at (-1,0.75) [circle,draw=black,fill=black,scale=0.5]{};
\node at (1,0.75) [circle,draw=black,fill=black,scale=0.5]{};

\draw (0+3,0) -- (-0.5+3,0.5);
\draw (0+3,0) -- (0.5+3,0.5);
\draw (-0.5+3,0.5) -- (-0.5+3,1);
\draw (0.5+3,0.5) -- (0.5+3,1);
\draw (-0.5+3,1) -- (0.5+3,1);
\draw (-0.5+3,0.5) -- (-1+3,0.75);
\draw (-0.5+3,1) -- (-1+3,0.75);
\draw (0.5+3,0.5) -- (1+3,0.75);
\draw (0.5+3,1) -- (1+3,0.75);
\draw (0,0) -- (3,0);

\node at (0+3,0) [circle,draw=black,fill=black,scale=0.5]{};
\node at (-0.5+3,0.5) [circle,draw=black,fill=black,scale=0.5]{};
\node at (0.5+3,0.5) [circle,draw=black,fill=black,scale=0.5]{};
\node at (-0.5+3,1) [circle,draw=black,fill=black,scale=0.5]{};
\node at (0.5+3,1) [circle,draw=black,fill=black,scale=0.5]{};
\node at (-1+3,0.75) [circle,draw=black,fill=black,scale=0.5]{};
\node at (1+3,0.75) [circle,draw=black,fill=black,scale=0.5]{};

\node at (-0.1,-0.1) {$u$};
\node at (-0.6,0.4) {$u_1$};
\node at (-1.1,0.65) {$x$};
\node at (0.65,0.4) {$u_2$};
\node at (1.1,0.65) {$y$};
\node at (-0.65,1.1) {$u_3$};
\node at (0.6,1.1) {$u_4$};
\node at (3.1,-0.1) {$v$};

\end{tikzpicture}
\end{center}
\caption{\label{graphnot1222} A non $(1,2,2,2)$-packing colorable graph.}
\end{figure}
\begin{proof}
Let $G$ be the graph depicted in Figure~\ref{graphnot1222}.  
We suppose that a $(1,2,2,2)$-packing coloring uses the colors $1,a_1 ,a_2$ and $a_3$, the meaning of which should be clear. 
Suppose, on the contrary, that $G$ has a $(1,2,2,2)$-packing coloring. Clearly, one vertex among $u$ and $v$ should be colored by a color different from $1$. 
Suppose, without loss of generality that $u$ has a color in $\{a_1 ,a_2,a_3\}$, say $a_1$. 
Note that at least one of the vertices $u_3$ and $u_4$ is not colored by color $1$, and assume without loss of generality that $u_3$ is colored with $a_2$. Consequently, we can only use colors $1$ and $a_3$ for the three vertices $u_2,u_4,y$ (forming a triangle), which is not possible. 
\end{proof}

Second, we prove that Theorem~\ref{th:1222} cannot be improved by replacing in $(1,2,2,2)$ an integer $2$ by an integer $3$, when $G$ is in the class of subcubic triangle-free outerplanar graphs.

\begin{prop}
There exists a subcubic triangle-free outerplanar graph, which is not $(1,2,2,3)$-packing colorable.
\label{prp:ex3}
\end{prop}
\begin{proof}
Let $G_{25}$ be the graph obtained from six copies of the 5-cycle, one of which we distinguish and denote by $C$; and to each vertex $x$ of $C$ we add an edge between $x$ and a vertex of its own copy of $C_5$. Note that in a $(1,2,2,3)$-packing coloring of $C_5$ one must color two vertices by color $1$, and each of the other three vertices receives its own color among $\{2,2',3\}$.  Each vertex $x$ of the central cycle $C$ is at distance at most $3$ from all vertices in the 5-cycle, which is attached to $x$. Thus, the assumption that $G_{25}$ is $(1,2,2,3)$-packing colorable implies that a vertex $x$ in $C$ is colored by color $3$. However, distances from $x$ to vertices of the 5-cycle attached to $x$ prevent the use of color $3$ in that copy of $C_5$, which is in contradiction to the existence of a $(1,2,2,3)$-packing coloring of $G_{25}$. 
\end{proof}

%%%%%%%
%%%%%%%%%%%%
%%%%  C O N C L U S I O N
%%%%%%%%%%%
%%%%%%%

\section{Concluding remarks}

Theorem~\ref{mainth} gives a partial (affirmative) answer to the question posed in several papers concerning the boundedness of the packing chromatic number in the class of planar subcubic graphs. Instead of repeating the question, we propose two problems that lie between Theorem~\ref{mainth} and this question. In one of them, we consider a non-bipartite extension of the theorem, and in the other we replace outerplanar graphs with planar graphs.

\begin{ques} 
Is the packing chromatic number bounded in the class of 2-connected subcubic outerplanar graphs?
\end{ques}

\begin{ques} 
Is the packing chromatic number bounded in the class of 2-connected bipartite subcubic planar graphs?
\end{ques}

While we do not dare to suggest an answer to the above questions, we strongly believe that Theorem~\ref{mainth} could be extended from the $2$-connected case to all bipartite subcubic outerplanar graphs. 

In Section~\ref{sec:1222}, we proved that a subcubic triangle-free outerplanar graph is $(1,2,2,2)$-packing colorable, and it is not $(1,2,2,3)$-packing colorable in general. A similar proof as in Theorem~\ref{th:1222} can be used to prove the following result.
\begin{theo}
\label{th:last}
If $G$ is a subcubic triangle-free outerplanar graph, then $G$ is $(1,1,2)$-packing colorable.
\end{theo}
The above result can be viewed as an extension of the $3$-colorability of outerplanar graphs. We cannot omit the triangle-free condition from Theorem~\ref{th:last}, as demonstrated by the following example. Take four copies of the triangle $C_3$, one of which we distinguish and denote it by $C$; and to each vertex $x$ of $C$ we add an edge between $x$ and a vertex of its own copy of $C_3$. It is easy to see that the resulting graph $G$ is not $(1,1,2)$-packing colorable (clearly, $G$ is outerplanar and subcubic). In addition, Theorem~\ref{th:last} cannot be improved in such a way that the integer $2$ be replaced by $3$ in $(1,1,2)$-packing colorability of subcubic triangle-free outerplanar graphs. To see this, take the graph $G_{25}$ from the proof of Proposition~\ref{prp:ex3} as an example.

\section*{Acknowledgments}
We are grateful to an anonymous referee for a careful reading of the initial version of the paper and for a number of suggestions that helped to improve the presentation.

This work was performed with the financial support of the bilateral project "Distance-constrained and game colorings of graph products" (BI-FR/18-19-Proteus-011).

B.B. acknowledges the financial support from the Slovenian Research Agency (research core funding No.\ P1-0297 and project Contemporary invariants in graphs No.\ J1-9109).

\end{document}